\documentclass[a4paper,abstracton]{scrartcl}
\usepackage{amsfonts}
\usepackage{amssymb}
\usepackage{amsmath}
\usepackage{amsthm}
\usepackage{comment}

\usepackage[ruled,vlined,linesnumbered]{algorithm2e}
\usepackage{hyperref}
\usepackage[noabbrev,capitalize]{cleveref}

\usepackage{graphicx}
\usepackage{color}
\usepackage{csquotes}
\usepackage{a4wide}

\usepackage{tikz}
\usetikzlibrary{calc}
\usetikzlibrary{patterns,decorations.pathreplacing}
\usetikzlibrary{arrows.meta}
\newtheorem{theorem}{Theorem}

\newtheorem{lemma}[theorem]{Lemma}

\newtheorem{cor}[theorem]{Corollary}

\usepackage{authblk}

\allowdisplaybreaks

\begin{document}

%%%%%%%%%%%%%%%%%%%%%%%%%%%%%%%%%%%%%%%%%%%%%%%%%%%%%%%%%%%%%%%%%
%%%%%%  Macros
%%%%%%%%%%%%%%%%%%%%%%%%%%%%%%%%%%%%%%%%%%%%%%%%%%%%%%%%%%%%%%%%%
\newcommand{\X}{{\mathcal{X}}}
\newcommand{\cU}{{\mathcal{U}}}
\newcommand{\cI}{{\mathcal{I}}}
\newcommand{\cC}{{\mathcal{C}}}
\newcommand{\cB}{{\mathcal{B}}}
\newcommand{\rev}[1]{\textcolor{blue}{#1}} %text segments changed in the revision
\newcommand{\R}{\mathbb{R}}
\newcommand{\N}{\mathbb{N}}

\newcommand{\radj}{R-Adj-SAT}

\newcommand{\unc}{\mathcal{Z}}

\newcommand{\set}[1]{\{ #1 \}}
\newcommand{\fromto}[2]{\set{#1,\dots,#2}}
%%%%%%%%%%%%%%%%%%%%%%%%%%%%%%%%%%%%%%%%%%%%%%%%%%%%%%%%%%%%%%%%%
% \title{Two-Stage Robust Optimization Problems with Post-Decision Uncertainty}
\title{On the Complexity of Robust Multi-Stage Problems in the Polynomial Hierarchy}

\author[1]{Marc Goerigk\footnote{Corresponding author. Email: marc.goerigk@uni-siegen.de}}
\author[2,3]{Stefan Lendl}
\author[2]{Lasse Wulf}

\affil[1]{Network and Data Science Management, University of Siegen,\authorcr Unteres Schlo{\ss} 3, 57072 Siegen, Germany}	
\affil[2]{Institute of Discrete Mathematics, Graz University of Technology,\authorcr Steyrergasse 30/II, 8010 Graz, Austria}
\affil[3]{Institute of Operations and Information Systems, University of Graz,\authorcr Universit\"atsstra{\ss}e 15, 8010 Graz, Austria}

\date{}

\maketitle

\begin{abstract}
We study the computational complexity of multi-stage robust optimization problems.
Such problems are formulated with alternating min/max quantifiers and therefore naturally fall into a higher stage of the polynomial hierarchy. Despite this, almost no hardness results with respect to the polynomial hierarchy are known. 

In this work, we examine the hardness of robust two-stage adjustable and robust recoverable optimization with budgeted uncertainty sets. Our main technical contribution is the introduction of a technique tailored to  prove $\Sigma^p_3$-hardness of such problems. We highlight a difference between continuous and discrete budgeted uncertainty: In the discrete case, indeed a wide range of problems becomes complete for the third stage of the polynomial hierarchy; in particular, this applies to the TSP, independent set, and vertex cover problems. However, in the continuous case this does not happen and problems remain in the first stage of the hierarchy. Finally, if we allow the uncertainty to not only affect the objective, but also multiple constraints, then this distinction disappears and even in the continuous case we encounter hardness for the third stage of the hierarchy. This shows that even robust problems which are already NP-complete can still exhibit a significant computational difference between column-wise and row-wise uncertainty.
\end{abstract}


\noindent\textbf{Acknowledgements:} Supported by the Deutsche Forschungsgemeinschaft (DFG) through grant GO 2069/1-1.

%%%%%%%%%%%%%%%%%%%%%%%%%%%%%%%%%%%%%%%%%%%%%%%%%%%%%%%%%%%%%%%%%%%%%%%%%
%%%%%%%%%%%%%%%%%%%%%%%%%%%%%%%%%%%%%%%%%%%%%%%%%%%%%%%%%%%%%%%%%%%%%%%%%

\section{Introduction}

\subsection{Motivation}

As single-stage robust optimization problems have been well studied, multi-stage problems have seen increasing attention in the robust optimization community. These are problems where the decision making process is split into two or even more stages. Examples for problems of this kind are robust two-stage adjustable optimization and robust recoverable optimization.

In this paper, we study this complexity of multi-stage robust optimization problems. The case of one-stage robust optimization has been thoroughly studied (see, e.g., \cite{kasperski2016robust}). In contrast, considerably less is known about the two-stage or recoverable robust case. In some cases, there is no theoretical foundation that excludes the possibility of a compact mixed-integer programming formulation (i.e., it is not clear if the problem is in NP or not). As many one-stage robust optimization problems are already NP-hard, it seems likely that a higher level of the complexity hierarchy \cite{stockmeyer1976polynomial} needs to be studied to capture the complexity of two-stage problems. Indeed, a compelling case has been made in \cite{woeginger2021trouble} that multi-stage complexity questions should be approached from this perspective. 

One very common claim in the literature is that even if the uncertainty 
is only present in the objective function (row-wise uncertainty), one can equivalently reduce this 
to an instance with uncertainty only in the constraints.
Hence, only algorithmic techniques for the more-general case of constraint-wise uncertainty are developed.
While this problem reduction is correct, we show that the 
more general case of uncertainty in the constraints leads to a jump of the problem complexity from 
NP-hardness to $\Sigma^p_3$-hardness. This implies that for the simpler case of row-wise uncertainty
one should be able to obtain more efficient algorithms than for the more general case of column-wise uncertainty.
Hence, the common reasoning to only develop algorithms for the more general case of uncertainty in the constraints is 
flawed and the study of specialized algorithms for the case of row-wise uncertainty is vital.

\subsection{Background}

Optimization problems in practice often contain parameters that cannot be known precisely. If we simply ignore this uncertainty and use estimated parameter values, we consider so-called nominal optimization problems. While such an approach may result in relatively small and easy to handle optimization problems, it may also result in high costs or infeasibility if parameters happen to deviate from the estimated value. For this reason, several approaches have been developed to include uncertainty already in the solution process; these include stochastic \cite{powell2019unified}, fuzzy \cite{lodwick2010fuzzy} and robust optimization \cite{ben2009robust}.

Robust optimization typically assumes that a set of possible scenarios can be constructed, but does not require a probability distribution over this set. In (classic, one-stage) robust optimization, we would like to find a solution that is feasible under every scenario and gives the best objective value with respect to the worst case scenario in the uncertainty \cite{ben2002robust}. Note that this description implies that uncertainty may be present in the constraints and in the objective function. As it is more convenient to avoid this distinction and to consider a unified setting, often the uncertain objective function is reformulated as an uncertain constraint in an epigraph formulation, which means that without loss of generality, we may consider the objective without uncertainty (see, for example, \cite{ben2002robust}).

This robust optimization approach of finding a solution that is feasible for all scenarios may be too conservative if the problem under consideration allows for more dynamic decision making. In two-stage (adjustable) robust optimization \cite{ben2004adjustable}, we distinguish between two types of variables. The values of here-and-now (first-stage) variables need to be decided beforehand. We then receive the information which scenario from the uncertainty set has been realized, before we decide on the value of wait-and-see (second-stage) variables. This means that the decision maker has a higher degree of flexibility and thus can find solutions with better objective value. This separation into first- and second-stage variables may reflect decisions that are made in different planning stages; e.g., the first-stage solution may reflect a long-term investment, while the second-stage variables may define how to operate it. A recent survey on adjustable robust optimization can be found in \cite{yanikouglu2019survey}.

A variant of this approach is recoverable robust optimization \cite{liebchen2009concept}. Here the second-stage variables reflect modifications to the first-stage solution, which we would like to keep as small as possible. As an example, consider a train timetable, where a solution needs to be communicated to travelers in advance, but still needs be adjusted to incorporate current delays during operation.

To formulate any robust optimization model, the uncertainty set is a central component. On the one hand, it needs to be flexible enough to reflect all possible scenarios; on the other hand, it should have a simple structure to improve the tractability of the resulting robust model. This has lead to a wide range of research into the formulation of uncertainty sets, see, e.g. \cite{bertsimas2009constructing}. Particularly successful in this trade-off have been budgeted uncertainty sets as originally introduced in \cite{bertsimas2003robust,bertsimas2004price}. The simple idea to construct such sets is to assume that all parameters are at their nominal values by default, but at most $\Gamma$ many values may deviate simultaneously within given intervals.

A high-level distinction can be made between discrete and continuous uncertainty sets. This also applies to budgeted uncertainty sets, where we either assume that a parameter deviates to an extreme value or not (discrete case); or where we may assume that more than $\Gamma$ parameters can deviate partially towards their extreme values (continuous case). In the case of one-stage optimization, it is well-known that the convex hull of an uncertainty set leads to the same robust optimization problem as when using the original uncertainty set \cite{yanikouglu2019survey}. This is different for two-stage optimization, where the possibility to react to a scenario means that there is difference between a discrete uncertainty set and its convex hull (see, e.g., \cite{chassein2018recoverable}).

A widely used method to solve two-stage robust optimization problems is to formulate a problem with a discrete subset of scenarios, and then to alternate between solving such reduced problems with solving the worst-case problem to find another scenario that is added to the current reduced problem \cite{aissi2009min,zeng2013solving}. 
A great advantage of this approach it can be easily applied to any two-stage robust problem. For some problems, it is possible to reformulate the two-stage problem using a compact mixed-integer programming formulation (i.e., the problem remains in NP). Often, such compact formulations outperform iterative approaches. Therefore, the choice of solution method is closely connected to the complexity class of the problem.

Even one-stage problems have been shown to be in higher complexity classes. The min-max regret knapsack problem with interval uncertainty was shown to be $\Sigma^p_2$-complete \cite{deineko2010pinpointing}, which also leads to $\Sigma^p_2$-completeness for the min-max regret weighted set covering problem \cite{coco2022robust}. Also the robust linear binary programming problem with binary uncertainty sets has been shown to be in the same complexity class \cite{claus2020note}.

Examples where the complexity of multi-stage robust problems has been studied include
shortest path \cite{busing2012recoverable}, spanning tree \cite{kasperski2011approximability}, or selection problems \cite{kasperski2017robust,goerigk2022recoverable}. In these cases, the analysis has focussed on NP-hardness. Only few results are available on higher levels of complexity. Indeed, a simple argument shows that if the uncertainty set is convex and only affects the objective function, the recoverable robust problem remains in NP under some mild assumptions (see \cite{hanasusanto2015k,buchheim2017min,bold2020}). In \cite{pfetsch2021generic}, a linear resilient design decision problem is presented, which is proved to be $\Sigma^p_3$-hard. In parallel to our own work, the complexity of recoverable robust problems with Hamming distance has been studied with similar results \cite{grune2022complexity}, albeit with what is called xor-dependency scenarios.

Related protection-interdiction problems \cite{nabli2022complexity} and bilevel problems \cite{caprara2016bilevel} are also known to be on higher levels of the polynomial hierarchy. Another way to consider general robust multi-stage is through quantified programming \cite{goerigk2021multistage}. It is known that the quantified program of $k$ stages is $\Sigma^p_{2k-1}$-hard, see \cite{chistikov2017complexity,nguyen2020computational}. With the availability of general multi-stage solvers \cite{phdhartisch}, more tools are at disposal to treat multi-stage robust problems.

\subsection{Contributions}

In its general form, the two-stage (adjustable) robust optimization is given by
\[ \min_{\pmb{x},\pmb{y}(\cdot)}  \left\{ \max_{\pmb{\zeta} \in \unc} \pmb{C}(\zeta)\pmb{x} + \pmb{c}(\zeta)\pmb{y}(\pmb{\zeta})  \colon \forall \pmb{\zeta} \in \unc \colon A(\pmb{\zeta}) \pmb{x} + B(\pmb{\zeta})\pmb{y}(\pmb{\zeta}) \leq \pmb{d}(\pmb{\zeta})  \right\} \]
where $\mathcal{Z}$ denotes the primitive uncertainty set that can influence coefficient matrices $A$ and $B$, the right-hand side $\pmb{d}$ and cost coefficients $\pmb{C}$ and $\pmb{c}$. While $\pmb{x}$-variables need to be fixed in advance, we can let $\pmb{y}$-variables depend on $\pmb{\zeta}$; equivalently, we can consider them as a function in $\pmb{\zeta}$. If there is no uncertainty in the constraints, we will also write $\X'$ for the set of feasible first-stage solutions, and $\X(\pmb{x})$ for the set of feasible second-stage solutions depending on first-stage solution $\pmb{x}$. In case of combinatorial optimization over a set $\X$, we often have that $\X'=\{0,1\}^n$ and $\X(\pmb{x}) = \{ \pmb{y}\in\{0,1\}^n : \pmb{x}+\pmb{y} \in\X\}$.

In the classic model introduced by Ben-Tal et al.~\cite{ben2004adjustable} both $\pmb{x} \in \mathbb{R}^{n_1}, \pmb{y} \in \mathbb{R}^{n_2}$ are continuous variables.
In this work we consider the computationally harder variant with (mixed) integer recourse, where there is the additional integrality constraint on some of the recourse decisions, hence $\pmb{y} = (\pmb{y}^c, \pmb{y}^{d}) \in (\mathbb{R}^{n_c}, \mathbb{Z}^{n_d})$.

There are different primitive uncertainty sets $\mathcal{Z} \subseteq \mathbb{R}^{\ell}$.
We focus on the discrete and continuous budgeted primitive uncertainty sets
\[ \mathcal{Z}^c =  \{ \pmb{\zeta} \in [0,1]^{\ell} \colon \| \pmb{\zeta} \|_1 \leq \Gamma \} \]
\[ \mathcal{Z}^d =  \{ \pmb{\zeta} \in \{0,1\}^{\ell} \colon \| \pmb{\zeta} \|_1 \leq \Gamma \} \]
with affine linear cost and right hand side functions $\pmb{c}(\cdot)$ and $\pmb{d}(\cdot)$. Such sets are particularly of interest for combinatorial optimization problems as the simplest, non-trivial shape of uncertainty. The image of parameters under the primitive uncertainty set is also called the uncertainty set and denoted by $\cU$. In the case of continuous budgeted uncertainty only in the second-stage objective coefficients $\pmb{c}$, we can write $\cU^c_\Gamma = \{\pmb{c} : \exists \zeta \in \mathcal{Z}^c \text{ s.t. } c_i = \underline{c}_i + (\overline{c}_i - \underline{c}_i)\zeta_i \text{ for all } i\}$ to denote the uncertainty set, and treat the case of discrete budgeted uncertainty set $\cU^d_\Gamma$ analogously.

The related approach of recoverable robustness \cite{liebchen2009concept} can be framed as a special case of this adjustable approach, which differs in philosophy. In the most frequent formulation of this problem, we assume that a full solution is constructed in the first stage, while the second stage allows for modifications to the first-stage solution where we bound by how much the solution is allowed to be changed.

In this paper we make the following contributions.
\begin{itemize}
\item Our main technical contribution is the introduction of a two-stage robust satisfiability problem 
and the proof of its $\Sigma^p_3$-completeness. 
This problem is tailored to allow reductions to two-stage robust optimization problems with budgeted uncertainty.
We use this technique for all of the following hardness results.
We also show that the $K$-stage problem is $\Sigma_{2K-1}^p$-complete for fixed $K$ (see Section~\ref{sec:radj}).

\item Using this result, we show that the two-stage robust optimization problem with (mixed) binary recourse and continuous budgeted uncertainty set affecting constraints is $\Sigma_3^p$-hard (see Section~\ref{sec:contbudgeted}). As any nominal problem that is in NP remains in NP when a continuous uncertainty set affects the objective function, this highlights a difference when we consider column-wise uncertainty in contrast to row-wise uncertainty.

\item Turning to discrete budgeted uncertainty sets, we show that the two-stage and recoverable robust independent set problems are $\Sigma_3^p$-hard (see Sections~\ref{subsec:two-stage-ind-set} and \ref{subsec:rec-ind-set}).

\item We further show that two-stage and recoverable robust independent set, traveling salesman, and vertex cover problems with discrete uncertainty sets are $\Sigma_3^p$-hard (see Sections~\ref{subsec:two-stage-tsp}, \ref{subsec:rec-tsp}, and \ref{subsec:vertexcover}).

\item Finally, we show that for fixed $K$, the $K$-stage independent set problem with discrete uncertainty is $\Sigma_{2K-1}^p$-hard. If $K$ is part of the input, then the problem becomes PSPACE-hard (see Section~\ref{subsec:multistage}).
\end{itemize}

We close with a summary of results and further research questions in Section~\ref{sec:conclusions}. A brief overview to the $\Sigma^p_3$-hardness results from this paper is given in Table~\ref{tab:overview}. Columns ''cont. b.u.'' and ''disc. b.u.'' correspond to continuous and discrete budgeted uncertainty, respectively. The case of right-hand side budgeted uncertainty in combinatorial problems is omitted, as the setting is not well-defined.

\begin{figure}[htbp]
\begin{center}
\begin{tabular}{r|ccc}
 & cont. b.u. & cont. b.u. & disc. b.u. \\
problem & RHS & objective & objective \\ 
\hline
MIP & $\Sigma^p_3$ & NP & $\Sigma^p_3$ \\
IS & - & NP & $\Sigma^p_3$ \\
TSP & - & NP & $\Sigma^p_3$ \\
VC & - & NP & $\Sigma^p_3$
\end{tabular}
\end{center}
\caption{$\Sigma^p_3$ complexity results for adjustable and recoverable problems in this paper.}\label{tab:overview}
\end{figure}

\section{Robust Adjustable SAT}
\label{sec:radj}

We present a variant of the satisfiability problem (SAT), which we call the \emph{robust adjustable SAT problem with budgeted uncertainty} ({\radj} for short) that is inspired by Lemma~2.3 in \cite{pfetsch2021generic}.
We show that this problem is $\Sigma_3^p$-hard. 
The problem is tailored to be similar to many problems in the setting of robust optimization with discrete budgeted uncertainty, specifically the setting of robust recoverable and robust two-stage optimization. 
In fact, every hardness proof in the remaining paper is based on {\radj}. 
We believe that also many other problems in the same setting admit a simple reduction from {\radj}.

We recall the following terms: 
A \emph{boolean variable} $x$ is a variable which takes either the value '0' or '1'. 
A \emph{literal} corresponding to $x$ is either the \emph{positive literal} $x$ or the \emph{negative literal} $\overline{x}$. A \emph{clause} is a disjunction of literals.
 A formula is in \emph{conjunctive normal form} if it is a conjunction of clauses. 
 An \emph{assignment} is a map $f : X \rightarrow \set{0,1}$, where $X$ is the set of all variables. 
 We are now ready to present the problem {\radj}.

\begin{quote}
Problem {\radj}
\\
\textbf{Instance:}  A SAT-formula $\varphi(x,y,z)$ given in conjunctive normal form. A partition of the set of variables into three disjoint parts $X \cup Y \cup Z$. An integer $\Gamma \geq 0$.  
\\
\textbf{Question:} Is there an assignment of the variables in $X$ such that for all subsets $Y' \subseteq Y$ of size $|Y'| \leq \Gamma$, if we set all variables in $Y'$ to '0', then there exists an assignment of the remaining variables $(Y \setminus Y') \cup Z$ such that $\varphi(x,y,z)$ is satisfied?
\end{quote}

The problem {\radj} can also be understood as a game between two players. 
First, player number one fixes the assignment of variables in $X$. 
Secondly, player number two chooses a subset $Y' \subseteq Y$ of size at most $\Gamma$. 
These variables are set to '0'. 
After that, player one chooses all remaining variables. 
The goal for player one is to satisfy the formula $\varphi(x,y,z)$ while player two has the opposite goal. 
The given instance of {\radj} is a Yes-instance if and only if player one has a winning strategy. 
We now wish to show that {\radj} is $\Sigma_3^p$-complete. 
In order to do that, we reduce from the canonical $\Sigma_3^p$-complete problem $\exists\forall\exists$-SAT \cite{stockmeyer1976polynomial}.
\begin{quote}
Problem $\exists\forall\exists$-SAT
\\
\textbf{Instance:}  A SAT-formula $\psi(a,b,c)$ given in conjunctive normal form. A partition of the set of variables into three disjoint parts $A \cup B \cup C$.
\\
\textbf{Question:} Is there an assignment of the variables in $A$ such that for all assignments of variables in $B$, there exists an assignment of the variables in $C$ such that $\psi(a,b,c)$ is satisfied?
\end{quote}

\begin{theorem}
\label{thm:gamma-SAT}
Problem {\radj} is $\Sigma_3^p$-complete, even if $|X|=|Y|=|Z|$ and all clauses in $\varphi$ contain exactly three literals.
\end{theorem}
\begin{proof}
We first prove the theorem without the additional assumptions that $|X|=|Y|=|Z|$ and that all clauses in $\varphi$ contain three literals. We explain at the end of the proof why these assumptions can be added.

 It is clear that {\radj} is contained in the class $\Sigma_3^p$.
So it remains to prove $\Sigma_3^p$-hardness. 
Assume we are given an instance $(\psi,A,B,C)$ of $\exists\forall\exists$-SAT, we construct in polynomial time an instance $(\varphi,X,Y,Z,\Gamma)$ of {\radj} such that $(\varphi,X,Y,Z,\Gamma)$ is a Yes-instance if and only  $(\psi,A,B,C)$ is a yes-instance. 
It is well-known \cite{schaefer2002completeness} that $\exists\forall\exists$-SAT is $\Sigma_3^p$-complete even if $|A| = |B| = |C|$ , so we can without loss of generality denote the variables in $A,B,C$ by $A = \fromto{a_1}{a_n}$, $B = \fromto{b_1}{b_n}$ and $C = \fromto{c_1}{c_n}$. 
We now define the sets $X,Y,Z$ of variables the following way:
\begin{align*}
X &= \fromto{x_1}{x_n}\\
Y &= \fromto{y^t_1}{y^t_n} \cup \fromto{y^f_1}{y^f_n}\\
Z &= \fromto{z_1}{z_n} \cup \fromto{s_1}{s_n} \cup \set{s}.
\end{align*}
This means that $X$ contains $n$ variables, $Y$ contains $2n$ variables, and $Z$ contains $2n+1$ variables. 
Furthermore, let $C_1,\dots, C_\ell$ be the clauses of $\psi$.
Note that these clauses use variables from $A \cup B \cup C$. 
We want to replace these clauses with new clauses which use variables from $X \cup Y \cup Z$. 
We define the \emph{replacement} of a literal in $A$ by $r(a_i) := x_i$ and $r(\overline a_i) := \overline x_i$ for $i=1,\dots,n$. 
Likewise, we define $r(c_i) := z_i$ and $r(\overline c_i) = \overline z_i$. 
In $B$, we define the slightly different replacement $r(b_i) := y^t_i$ and $r(\overline b_i) := y^f_i$. 
The formula $\varphi$ contains the following three types of clauses: 
\begin{itemize}
\item For every clause $C_i = w_1 \lor \dots \lor w_t$ contained in $\psi$, where $w_1 , \dots, w_t$ are its literals, we add the following clause $r(C_i)$ of length $t+1$ to $\varphi$:
\begin{equation}
r(C_i) := r(w_1) \lor \dots r(w_t) \lor s. \label{clauses1}
\end{equation}
\item For every $i=1,\dots,n$, the formula $\varphi$ contains the two clauses
\begin{equation}
\overline s_i \lor y^t_i \text{ and } \overline s_i \lor y^f_i. \label{clauses2}
\end{equation}
\item The formula $\varphi$ contains the single clause 
\begin{equation}
\overline s \lor s_1 \lor \dots \lor s_n. \label{clauses3}
\end{equation}
\end{itemize}
Finally, we let $\Gamma :=  n$. This completes our description of the new instance. We claim that the new instance $(\varphi,X,Y,Z, \Gamma)$ is a Yes-instance if and only if $(\psi,A,B,C)$ is a Yes-instance.

Assume that $(\psi,A,B,C)$ is a Yes-instance. This means there is an assignment $f_A : A \rightarrow \set{0,1}$, such that for all assignments $f_B : B \rightarrow \set{0,1}$ there is an assignment $f_C: C \rightarrow \set{0,1}$ such that $\psi(a,b,c)$ is satisfied. 
We now show how to satisfy $\varphi(x,y,z)$. 
First, let $g_X : X \rightarrow \set{0,1}$ be the assignment which corresponds to $f_A$, that is, variable $x_i=1$ if and only if variable $a_i=1$. 
Next, let $Y' \subseteq Y$ be an arbitrary subset of $Y$ of size $|Y'| \leq \Gamma = n$. 
We have to show that the assignment $g_X$ can be completed to a satisfying assignment such that all variables in $Y'$ are assigned '0'.  We distinguish three cases:

\textbf{Case 1:} There is an $i \in \fromto{1}{n}$ such that both $y_i^t,y_i^f \in Y'$: Then because $|Y'| \leq n$, there is another $j \neq i$ such that neither of $y^t_j, y^f_j$ is contained in $Y'$. 
We now show how to complete the partial assignment $g_x$ to a satisfying assignment: We set all the variables in $Y'$ to 0, and furthermore $y^t_j=y^f_j=1$ and $s_j=1$ and $s_k=0$ for $k \neq j$ and $s=1$. 
All remaining variables in $Y \cup Z$ are set arbitrarily. Note that then all the clauses \eqref{clauses1}, \eqref{clauses2} and \eqref{clauses3} are satisfied and all variables in $Y'$ are set to 0, as requested.

\textbf{Case 2:} We have $|Y'| < n$. Then again there is an index $j$ such that neither of $y^t_j, y^f_j$ is contained in $Y'$. This case is analogous to case 1.

\textbf{Case 3:} We have $|Y'| = n$ and for every $i=1,\dots,n$, exactly one of $y_i^t$, $y^f_i$ is contained in $Y'$. 
We then consider the assignment $f_B : B \rightarrow \set{0,1}$, where for all $i=1,\dots,n$ we have $f_b(b_i) = 0$ if $y_i^t \in Y'$ and $f_b(b_i) = 1$ if $y_i^f \in Y'$. 
In other words, under this assignment we have that $b_i = 1$ if and only if $y^t_i$ is not forced to 0, i.e.\ if $y^t_i \not\in Y'$. 
By assumption, there exists an assignment $f_C : C \rightarrow \set{0,1}$, such that under assignments $f_A, f_B, f_C$ the formula $\psi$ is satisfied. 
We define the assignment $g_Y : Y \rightarrow \set{0,1}$ such that $g_Y(y) = 0$ if $y \in Y'$ and $g_Y(y) = 1$ if $y\not\in Y'$ for all $y \in Y$. 
We furthermore define the assignment $g_Z : Z \rightarrow \set{0,1}$ such that $g_Z(z_i) = 1$ if and only if $f_C(c_i) = 1$. 
We also let $g_Y(s_1) = \dots = g_Y(s_n) = g_Y(s) = 0$. 
It follows from the definition of the function $r(\cdot)$ and the properties of the assignment $f_A,f_B,f_C$ that all clauses \eqref{clauses1}, \eqref{clauses2} and \eqref{clauses3} are satisfied. This was to show.

For the reverse direction, we have to show that if $(\varphi,X,Y,Z,\Gamma)$ is a Yes-instance, then $(\psi,A,B,C)$ is a Yes-instance. 
Assume that $(\varphi,X,Y,Z,\Gamma)$ is a Yes-instance, then there is an assignment $g_X : X \rightarrow \set{0,1}$ such that for all $Y' \subseteq Y$ of size $|Y'| \leq \Gamma = n$ the assignment can be completed to a satisfying assignment such that all variables in $Y'$ are '0'.
Let $f_A : A \rightarrow \set{0,1}$ be the assignment that corresponds to $g_X$, that is, $f_A(a_i) = 1$ if and only if $g_X(x_i) = i$ for all $i=1,\dots,n$.
 Let $f_B : B \rightarrow \set{0,1}$ be an arbitrary assignment. 
 We have to show that there is an assignment $f_C : C \rightarrow \set{0,1}$ such that $\psi(a,b,c)$ is satisfied. 
 In order to do this, define the set $Y' := \set{y_i^t : f_B(x_i) = 0} \cup \set{y_i^f : f_B(x_i) = 1}$. 
 Note that $|Y'| = \Gamma$. By the properties of the Yes-instance of {\radj}, there exist assignments $g_Y,g_Z$ of variables in $Y,Z$ such that all variables in $Y'$ are 0 and such that $\varphi$ is satisfied under $(g_X,g_Y,g_Z)$. 
 In particular, the clauses \eqref{clauses2} are satisfied. For every $i=1,\dots,n$, observe that $Y'$ contains either $y_i^f$ or $y_i^t$, and therefore we have $s_i=0$ in assignment $g_Z$. 
 It follows from \eqref{clauses3} that $s=0$. 
 We now define an assignment $f_C : C \rightarrow \set{0,1}$ by letting $f_C(c_i) = 1$ if and only $g_Z(z_i) = 1$ for all $i=1,\dots,n$. 
 We claim that under the assignment $(f_A,f_B,f_C)$ the formula $\varphi$ is satisfied. 
 To see this, observe that clause $r(C_i)$ corresponds to clause $C_i$ and recall that $s=0$. 
 Now, if clause $r(C_i)$ is satisfied by some variable $x_j$ or $z_j$, it is clear that also $C_i$ is satisfied because the assignments $f_A$ and $g_X$ ($f_C$ and $g_Z$ respectively) correspond to each other. 
 If $r(C_i)$ is satisfied by some variable $y^t_i$, then this implies $y^t_i \not\in Y'$ and therefore $f_B(b_i)=1$ and so $C_i$ is satisfied.
  Analogously if $r(C_i)$ is satisfied by some variable $y^f_i$, then $y^t_i \not\in Y'$ and $f_B(b_i)=0$ and $C_i$ is satisfied. 
  This shows that the whole formula $\varphi$ is satisfied. 
  This completes the proof.
  
Finally, we show why one can make the additional assumption that $|X|=|Y|=|Z|$ and that all clauses in $\varphi$ contain three literals: One can make use of a standard trick which transforms a clause of arbitrary length into a set of clauses of length exactly 3 by introducing additional helper variables \cite{garey1979computers}.
We apply this trick to all clauses of $\varphi$ and add the resulting helper variables into the set $Z$. After that we can fill up the sets $X,Y,Z$ with "useless" new variables that do not appear in any clause, until we have $|X| = |Y| = |Z|$. The rest of the proof proceeds in the same manner.
\end{proof}

Finally, we present a multi-stage version of {\radj}. Let $k \geq 1$ be an integer.

\begin{quote}
Problem $k$-stage {\radj}
\\
\textbf{Instance:}  A SAT-formula $\varphi(x,y,z)$ given in conjunctive normal form. A partition of the set of variables into $2k-1$ disjoint parts $X_1 \cup \dots X_{2k-1}$. An integer $\Gamma \geq 0$.  
\\
\textbf{Question:} Is there an assignment of the variables in $X_1$ such that for all subsets $X_2' \subseteq X_2$ of size $|X_2'| \leq \Gamma$, there exists an assignment of $X_2 \cup X_3$ with all of $X'_2$ assigned '0' such that for all subsets $X_4' \subseteq X_4$ of size $|X_4'| \leq \Gamma$ there exists an assignment of $X_4 \cup X_5$ with all of $X_4'$ assigned '0', etc\dots such that $\varphi$ is satisfied?
\end{quote}

The following theorem can be proven by adapting the proof of \cref{thm:gamma-SAT}.

\begin{theorem}
\label{thm:multi-stage-gamma-sat}
If $k \geq 1$ is a constant, then $k$-stage {\radj} is $\Sigma_{2k-1}^p$-complete. If $k$ is part of the input, then the problem is PSPACE-complete. This holds even if $|X_1|=\dots=|X_{2k-1}|$ and all clauses of $\varphi$ contain exactly three literals.
\end{theorem}

\section{Continuous Budgeted Uncertainty}
\label{sec:contbudgeted}

In this section, we study adjustable robust optimization problems with 
continuous budgeted uncertainty $\mathcal{Z}^{c}$ and general mixed integer programming constraints (MIP).

% The adjustable robust problem set problem is the robust problem of choosing an independent set in two stages: First choose a partial independent set $I_1$ in the first stage, then after the reveal of the uncertain cost function $c$ choose the remainder $I_2$ of the independent set. We wish to maximize the cost $c(I_1 \cup I_2)$ in the worst case. Observe that this is a maximization problem, in contrast to many other problems considered in robust optimization which are minimization problems.

%  The problem of finding a robust two-stage independent set is formally defined as

% \begin{equation*}
% \textsc{Rob} = \max_{\pmb x \in \set{0,1}^V} \min_{\pmb c \in \cU_\Gamma } \max_{\pmb y \in \X(\pmb x)} \pmb C \pmb x + \pmb c \pmb y, 
% \end{equation*} 

% where $G = (V,E)$ denotes the input graph, $\pmb C \in \R_{\geq 0}^V$ denotes the first-stage costs, $\X$ denotes the set of all binary indicator vectors of independent sets in that graph, and $\X(\pmb x) = \set{y \in \set{0,1}^V \mid \pmb x + \pmb y \in \X}$ denotes the set of all second-stage solutions $\pmb y$ such that $\pmb y$ together with $\pmb x$ forms an independent set. To treat the case that a first-stage vector $\pmb x$ is selected which can not be completed to an independent set, i.e. $\X(\pmb x) = \emptyset$ we define $\max \emptyset = -\infty$. This means the solution has the objective value $-\infty$ in that case.

We recall the following observation from \cite{bold2020}. 

\begin{theorem}\label{th1}
Let an adjustable robust problem with uncertainty in the objective of the form
\[ \min_{\pmb{x}\in \X'} \max_{\pmb{\zeta}\in\mathcal{Z}} \min_{\pmb{y}\in\X(\pmb{x})} f(\pmb{x},\pmb{y},\pmb{\zeta}) \]
be given with a compact convex set $\mathcal{Z}$ and a function $f$ that is linear in $\pmb{y}$ and concave in $\pmb{\zeta}$. Then this problem is equivalent to
\[ \min_{\pmb{x}\in\X'} \min_{\pmb{y}^{(1)},\ldots,\pmb{y}^{(n+1)}\in\X(\pmb{x})} \max_{\zeta\in\mathcal{Z}} \min_{i\in[n+1]} f(\pmb{x},\pmb{y}^{(i)},\pmb{\zeta}). \]
\end{theorem}
\begin{proof}
Using Carath\'eodory's theorem and the minimax theorem, we conclude the following equalities, where $\Delta_n$ denotes the $n$-simplex.
\begin{align*}
&\min_{\pmb{x}\in\X'} \max_{\zeta\in\mathcal{Z}}\min_{\pmb{y}\in\X(\pmb{x})} f(\pmb{x},\pmb{y},\pmb{\zeta}) \\
= & \min_{\pmb{x}\in\X'} \max_{\zeta\in\mathcal{Z}}\min_{\pmb{y}\in conv(\X(\pmb{x}))} f(\pmb{x},\pmb{y},\pmb{\zeta}) \\
= & \min_{\pmb{x}\in\X'} \min_{\pmb{y}\in conv(\X(\pmb{x}))} \max_{\zeta\in\mathcal{Z}} f(\pmb{x},\pmb{y},\pmb{\zeta}) \\
= & \min_{\pmb{x}\in\X'} 
\min_{\pmb{y}^{(1)},\ldots,\pmb{y}^{(n+1)}\in\X(\pmb{x})} \min_{\pmb{\lambda}\in\Delta_{n+1}} 
\max_{\zeta\in\mathcal{Z}} f(\pmb{x}, \sum_{i=1}^{n+1}\lambda_i\pmb{y}^{(i)},\pmb{\zeta}) \\
= & \min_{\pmb{x}\in\X'} \min_{\pmb{y}^{(1)},\ldots,\pmb{y}^{(n+1)}\in\X(\pmb{x})}  
\max_{\zeta\in\mathcal{Z}} \min_{\pmb{\lambda}\in\Delta_{n+1}} f(\pmb{x}, \sum_{i=1}^{n+1}\lambda_i\pmb{y}^{(i)},\pmb{\zeta}) \\
= & \min_{\pmb{x}\in\X'} \min_{\pmb{y}^{(1)},\ldots,\pmb{y}^{(n+1)}\in\X(\pmb{x})}  
\max_{\zeta\in\mathcal{Z}} \min_{\pmb{\lambda}\in\Delta_{n+1}} \sum_{i=1}^{n+1}\lambda_i f(\pmb{x}, \pmb{y}^{(i)},\pmb{\zeta}) \\
= & \min_{\pmb{x}\in\X'} \min_{\pmb{y}^{(1)},\ldots,\pmb{y}^{(n+1)}\in\X(\pmb{x})}  
\max_{\zeta\in\mathcal{Z}} \min_{i\in[n+1]} f(\pmb{x}, \pmb{y}^{(i)},\pmb{\zeta}) 
\end{align*}
\end{proof}

\begin{cor}
Under the assumptions of Theorem~\ref{th1}, two-stage and recoverable robust problems with continuous budgeted uncertainty are in NP.
\end{cor}
\begin{proof}
As $f(\pmb{x},\pmb{y},\pmb{\zeta})$ is concave in $\pmb{\zeta}$, the function $\min_{i\in[n+1]} f(\pmb{x},\pmb{y}^{(i)},\pmb{\zeta})$ remains concave in $\pmb{\zeta}$. Hence, the adversary problem is to maximize a concave function over a compact convex set. As it is possible to separate continuous budgeted uncertainty sets, the adversary problem can be solved in polynomial time \cite{grotschel1981ellipsoid}. Hence, it is possible to give a certificate of polynomial size that can be checked in polynomial time, which means that the two-stage problem is in NP. As the recoverable problem can be framed as a special case (compare Theorem~\ref{th1} with the result in \cite{bold2020}), this result also holds in this case.
\end{proof}

\begin{comment}
\begin{lemma}\label{lemma:makebinary}
    The optimization problem 
    \[ \max_{\zeta \in [0,1]} \min \{ y \colon y \geq \zeta, y \geq 1-\zeta, y \in \mathbb{R} \} \]
    has $\zeta \in \{0,1\}$ as optimal solutions.
\end{lemma}

\begin{lemma}\label{lemma:makebinary}
    The optimization problem 
    \[ \max_{\zeta \in [0,1]} \min \{ -y \colon y \leq \zeta, y \leq 1-\zeta, y \in \mathbb{R} \} \]
    has $\zeta \in \{0,1\}$ as optimal solutions.
\end{lemma}
\end{comment}

% \begin{lemma}\label{lemma:makebinary}
%     The optimization problem 
%     \[ \max_{\zeta \in [0,1]} \min \{ y \colon y \geq \zeta, y \geq 1-\zeta, y \in \mathbb{R} \} \]
%     has $\zeta \in \{0,1\}$ as optimal solutions.
% \end{lemma}

While this case of row-wise uncertainty remains in NP, we now show that 
with column-wise uncertainty (in the right-hand-side vector), 
the computational complexity of the adjustable robust problem can jump to a higher level, 
even if the uncertainty set remains a continuous budgeted set.

For the sake of this result we we consider the feasibility variant of the adjustable robust mixed integer programming problem with (mixed) binary recourse
and continuous budgeted uncertainty in the right-hand-side vector of the constraints.
Formally, we consider the problem
\begin{align*}
\exists\;\pmb{x}\; \forall\;\zeta \in \mathcal{Z}\; \exists\;\pmb{y} \colon & A\pmb{x} + B\pmb{y} \leq \pmb{d}(\pmb{\zeta})\\
& \pmb{x} = (\pmb{x}^c, \pmb{x}^d) \in (\mathbb{R}^{m_c}, \mathbb{Z}^{m_d})\\
& \pmb{y} = (\pmb{y}^c, \pmb{y}^d) \in (\mathbb{R}^{n_c}, \mathbb{Z}^{n_d}),
\end{align*}
where $A, B, d(\cdot)$ are the coefficient matrices and right-hand side vector
of a general linear mixed integer program.

\begin{theorem}
\label{thm:mixed-binary-recourse}
The feasibility variant of the adjustable robust mixed integer programming problem with (mixed) binary recourse and continuous budgeted uncertainty in the right-hand-side vector of the constraints is $\Sigma_3^p$-hard.
\end{theorem}

\begin{proof}
    We prove the claim by giving an reduction from the $\Sigma_3^p$-complete
    {\radj} problem. 
    Let an instance of {\radj} be given by a formula $\varphi$ in conjunctive normal form 
    on the variable set $X \cup Y \cup Z$ with $X=\{x_1, \dots, x_n\}$, $Y=\{y_1,\dots,y_n\}$ and $Z=\{z_1,\dots,z_n\}$,
    and a parameter $\Gamma \geq 0$. 
    Let $n = |X| = |Y| = |Z|$ be the number of variables in each set and let $\ell \in \N$ be the number of clauses in $\varphi$. 
    
    We construct an instance of the adjustable robust problem with binary recourse and continuous budgeted uncertainty
    in the constraints.
    The core of this construction is the straightforward reduction from 3SAT to the feasibility problem of MIPs.
    Hence we construct variable vectors $\pmb{x}', \pmb{y}'(\cdot), \pmb{z}'(\cdot)$ consisting of binary variables $x'_i, y'_i(\cdot), z'_i(\cdot)$ for 
    $i=1,\dots,n$ that correspond to the variables $x_i,y_i,z_i$ of the {\radj} instance.
    In our reduction we use a continuous budgeted uncertainty of $\mathcal{Z}^c = \{\pmb{\zeta}' \in [0,1]^n \colon \| \pmb{\zeta}' \| \leq \Gamma'\}$ of dimension $n$ with $\Gamma' = \Gamma$.
    The variables $\pmb{x}'$ correspond to the first stage decision and the variables $(\pmb{y}'(\cdot), \pmb{z}'(\cdot))$
    correspond to the second stage decision and depend on the uncertainty $\pmb{\zeta}'$.
    For convenience of notation we will omit this explicit dependence on $\pmb{\zeta}$ and just write $\pmb{y}'$ and $\pmb{z}'$ instead of $\pmb{y}'(\zeta)$ and $\pmb{z}'(\zeta')$.
    We define a replacement function $r$ that transforms each literal of $\varphi$ into a linear function using its corresponding binary variable.
    For each $i=1,\dots,n$ and literal $x_i$ or $\bar{x_i}$ of variables from $X$ the function $r$ is defined as 
    $r(x_i) = x'_i$ and $r(\bar{x}_i) = 1-x'_i$.
    Similarly, for each $i=1,\dots,n$ and literals $y_i, \bar{y}_i, z_i, \bar{z}_i$ we define 
    $r(y_i) = y'_i$, $r(\bar{y}_i) = 1-y'_i$, $r(z_i) = z'_i$ and $r(\bar{z}_i) = 1-z'_i$.
    Based on that, for each clause 
    \[ l_1 \lor l_2 \lor \dots \lor l_k \]
    in $\varphi$ we add the linear constraint
    \[ r(l_1) + r(l_2) + \dots + r(l_k) \geq 1 \] 
    to our new instance. Note that by identifying true and false assignments of $x_i \in X$ with $0$ and $1$ 
    assignments of $x'_i$ for $i=1,\dots,n$ and similarly for $Y$, $\pmb{y}'$ and $Z$, $\pmb{z}'$ variables 
    it holds that the feasible assignments for the clause $l_1 \lor l_2 \lor \dots \lor l_k$ are in one to one 
    correspondence with the feasible assignments for $r(l_1) + r(l_2) + \dots + r(l_k) \geq 1$.

    The main technical challenge is to encode the binary decisions of the adversary 
    in {\radj} using the continuous budgeted uncertainty.
    We need to model the fact that the adversary can force up to $\Gamma$ variables from $Z$ to false.
    To achieve this we add the constraints
    \[ z'_i \leq 2 - \varepsilon - \zeta_i \]
    for $i=1,\dots,n$ with $\varepsilon := \frac{1}{n \Gamma}$ to our instance.
    Observe, that this is the only place where the uncertainty $\pmb{\zeta}$ appears in our 
    reduction and it it appears in affine-linear form as part of the right-hand side of these constraints.
    Note, that this constraint only affects the second stage decision for $z'_i$ if 
    $\zeta_i > (1-\varepsilon)$. In this case the value of $z'_i$ is forced to be $0$. Hence, the adversarial decision to set $z'_i$ to a value larger 
    than $(1-\varepsilon)$ is in one to one correspondence with forcing $z_i$ to false.
    Given any attack $S$ of up to $\Gamma$ variables in the original {\radj} instance, 
    it is trivial that the corresponding set of $\pmb{z}$ variables can be attacked by 
    setting $\zeta_i = 1$ for $i \in S$. Clearly, for such an assignment we have $\| \zeta \|_1 = |S| \leq \Gamma'$ holds. 
    What remains to show is that the adversary can attack at most $\Gamma$ distinct 
    variables from $\pmb{z}'$ like that.
    Since $\varepsilon < \frac{1}{n \Gamma}$ the adversary has to invest $ > \frac{n^2-1}{n \Gamma}$ from its budget $\Gamma'$ for each $i$ to
    force $z'_i$ to $0$.
    Since after investing $\Gamma$ times $ > \frac{n^2-1}{n^2}$ the adversary has invested $> \frac{n^2-1}{n}$,
    leaving a remaining budget of $< n - \frac{n^2-1}{n} = \frac{1}{n} < \frac{n^2-1}{n \Gamma}$, since $\Gamma < n$.

    This concludes the proof of the $\Sigma^p_3$-hardness.
\end{proof}

% Note that this is not in contradiction to remark 1, as we have uncertainty in the constraints. 

\section{Discrete Budgeted Uncertainty}
\label{sec:discbudgeted}

In this section, we consider \emph{discrete budgeted uncertainty} for robust multi-stage versions of some famous classical problems, namely the traveling salesman problem (TSP), the independent set problem, and the vertex cover problem. We consider both \emph{two-stage adjustable} robustness, as well as \emph{recoverable} robustness. We prove that all considered problems are $\Sigma_3^p$-complete. The results also generalize from two-stage robust problems to $k$-stage robust problems with $k > 2$ (we obtain $\Sigma^p_{2k-1}$-hardness then). As a consequence, under common hardness assumptions all of the problems above can not be expressed as a polynomial-sized mixed integer program. This makes it very challenging for existing MIP-solvers to tackle these problems and so new techniques likely need to be developed. 

This is in contrast to Section~\ref{sec:contbudgeted}, where we studied the same robust problems under \emph{continuous} budgeted uncertainty and showed that they were contained in the class NP, i.e.\ not $\Sigma_3^p$-complete. 

We remark that all hardness proofs in this section work basically the same way: For each problem, we consider its classical reduction from SAT which was initially used to show NP-hardness, and modify this reduction to work with {\radj} instead. Informally speaking, this modification is especially easy if the classical reduction is composed out of variable gadgets and clause gadgets (i.e.\ parts of the reduced instance which mimic the behavior of the variables and the clauses from the original SAT instance). As this is a very general approach, we believe it can also be adapted to other problems.

The remainder of this section is structured as follows: In \cref{subsec:two-stage-ind-set,subsec:rec-ind-set} we consider two-stage and recoverable independent set. In \cref{subsec:two-stage-tsp,subsec:rec-tsp} we consider two-stage and recoverable . In \cref{subsec:vertexcover} we consider two-stage and recoverable vertex cover. Finally, in \cref{subsec:multistage}, we consider generalizations to multi-stage problems.

\subsection{Robust Two-Stage Independent Set}
\label{subsec:two-stage-ind-set}

The robust two-stage independent set problem is the robust problem of choosing an independent set in two stages: First choose a partial independent set $I_1$ in the first stage, then after the reveal of the uncertain cost function $c$ choose the remainder $I_2$ of the independent set. We wish to maximize the cost $c(I_1 \cup I_2)$ in the worst case. Observe that this is a maximization problem, in contrast to many other problems considered in robust optimization which are minimization problems.

 The problem of finding a robust two-stage independent set is formally defined as

\begin{equation*}
\textsc{Rob} = \max_{\pmb x \in \set{0,1}^V} \min_{\pmb c \in \cU_\Gamma } \max_{\pmb y \in \X(\pmb x)} \pmb C \pmb x + \pmb c \pmb y, 
\end{equation*} 

where $G = (V,E)$ denotes the input graph, $\pmb C \in \R_{\geq 0}^V$ denotes the first-stage costs, $\X$ denotes the set of all binary indicator vectors of independent sets in that graph, and $\X(\pmb x) = \set{y \in \set{0,1}^V \mid \pmb x + \pmb y \in \X}$ denotes the set of all second-stage solutions $\pmb y$ such that $\pmb y$ together with $\pmb x$ forms an independent set. To treat the case that a first-stage vector $\pmb x$ is selected which can not be completed to an independent set, i.e. $\X(\pmb x) = \emptyset$ we define $\max \emptyset = -\infty$. This means the solution has the objective value $-\infty$ in that case.

Finally, for given constants $\underline{c}_i \geq 0$ and $d_i \leq 0$ for all $i \in V$ and some integer $\Gamma \geq 0$, the set $\cU_\Gamma$ of uncertain cost functions is defined as
\[
\cU_\Gamma = \set{ \pmb c \in \R^V : c_i = \underline{c}_i + \delta_id_i, \delta_i \in \set{0,1}\ \forall i \in V,\ \sum_{i \in V}\delta_i \leq \Gamma }.
\]
In other words, $\cU_\Gamma$ contains those cost functions where at most $\Gamma$ entries deviate from their nominal cost $\underline{c}_i$. Note that because we have a maximization problem, we have $d_i \leq 0$. Finally, we define $\overline{c}_i = \underline{c}_i + d_i$.

\tikzstyle{vertex}=[draw,circle,fill=black, minimum size=4pt,inner sep=0pt]
\tikzstyle{edge} = [draw,-]
\tikzset{
cross/.style={path picture={ 
  \draw[black]
(path picture bounding box.south east) -- (path picture bounding box.north west) (path picture bounding box.south west) -- (path picture bounding box.north east);
}}
}
\tikzstyle{XORGadget}=[draw,circle,cross,minimum width=0.5,fill=white]
\tikzstyle{XOREdge}=[edge,rounded corners,{Stealth}-{Stealth}]
\begin{figure}[thpb]
\centering
\begin{tikzpicture}[scale=1,auto]
\node[vertex] (x3) at (0,10) {};
\node[above] at (x3) {$x_3$};
\node[vertex] (notx3) at (2,10) {};
\node[right] at (notx3) {$\overline x_3$};
\draw[edge] (x3) to (notx3);

\node[vertex] (x2) at (0,11) {};
\node[above] at (x2) {$x_2$};
\node[vertex] (notx2) at (2,11) {};
\node[above] at (notx2) {$\overline x_2$};
\draw[edge] (x2) to (notx2);

\node[vertex] (x1) at (0,12) {};
\node[above] at (x1) {$x_1$};
\node[vertex] (notx1) at (2,12) {};
\node[above] at (notx1) {$\overline x_1$};
\draw[edge] (x1) to (notx1);

\node[vertex] (c1) at (5,12) {};
\node[vertex] (c2) at (4,11.5) {};
\node[vertex] (c3) at (5,11) {};
\draw[edge] (c1) to (c2) to (c3) to (c1);
\draw[edge] (notx1) to (c1);
\draw[edge] (notx2) to (c2);
\draw[edge] (x3) to (c3);

\node[below=10] at (c3) {$x_1 \lor x_2 \lor \overline x_3$};
\end{tikzpicture}
\caption{Classical reduction of 3SAT to the max independent set problem.}
\label{fig:max-clique-classic-reduction}
\end{figure}
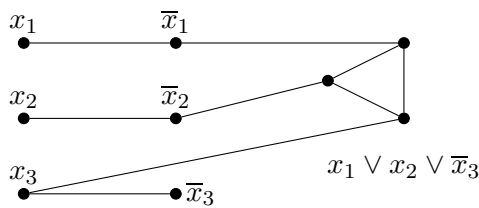
\begin{theorem}
\label{thm:discr-two-stage-ind-set}
Robust two-stage independent set with discrete budgeted uncertainty is $\Sigma_3^p$-complete.
\end{theorem}
\begin{proof}
It follows directly form the definition that the problem is contained in the class $\Sigma_3^p$. So it remains to prove $\Sigma_3^p$-hardness. The starting point is the classical reduction of 3SAT to the independent set problem \cite{garey1979computers}. An example is depicted in \cref{fig:max-clique-classic-reduction}. A \emph{variable-choice gadget} consists out of two vertices $x$ and $\overline x$ and an edge between them. For each variable in the 3SAT instance there is a variable-choice gadget. For each clause $x_1 \lor x_2 \lor x_3$ in the 3SAT instance, there is a \emph{clause gadget}, which consists out of a triangle, and a matching between the three vertices of the triangle and the three vertices corresponding to the opposite literals $\overline x_1$, $\overline x_2$ and $\overline x_3$. It is not hard to see that the 3SAT instance can be satisfied if and only if this graph contains an independent set of size $n+\ell$, where $n$ is the number of variables and $\ell$ is the number of clauses in the 3SAT instance. Indeed, vertex $x$ ($\overline x$, respectively) belongs to the independent set, if and only if in the satisfying assignment the variable $x$ is set to true (false, respectively). For every 3SAT formula $\psi$, let $G(\psi)$ denote the corresponding graph we just described. 

In order to prove the $\Sigma_3^p$-hardness of robust two-stage independent set, we modify this reduction to work with {\radj} instead of SAT. 
Let an instance of {\radj} be given by a formula $\varphi$ in conjunctive normal form on the variable set $X \cup Y \cup Z$, and a parameter $\Gamma \geq 0$. Let $n = |X| = |Y| = |Z|$ be the number of variables in each set and let $\ell \in \N$ be the number of clauses in $\varphi$. 
By \cref{thm:gamma-SAT}, this problem is $\Sigma_3^p$-complete. 
We construct an instance of robust independent set, consisting out of a graph, vertex costs $C, \underline{c}, \overline{c}$ and a parameter $\Gamma'$. The reduction is sketched in \cref{fig:max-clique-reduction}. Formally, it is described the following way: 
The graph of the instance is the graph $G(\varphi) =: (V,E)$. 
Observe that this graph has two kinds of vertices: 
First the $6n$ vertices $\fromto{x_1}{x_n} \cup \fromto{\overline{x}_1}{\overline{x}_n} \cup \fromto{y_1}{y_n} \cup \fromto{\overline{y}_1}{\overline{y}_n} \cup \fromto{z_1}{z_n} \cup \fromto{\overline{z}_1}{\overline{z}_n}$ corresponding to the $6n$ literals and secondly the rest of the vertices, corresponding to the clause gadgets. 
We define the vertex sets $V_1 := \fromto{x_1}{x_n} \cup \fromto{\overline{x}_1}{\overline{x}_n}$ and $V_2 := \fromto{y_1}{y_n}$. 
Depending on whether a vertex is in $V_1$, in $V_2$, or neither of these two, we define its costs $C, \underline{c}, \overline{c}$ as specified in \cref{table:independent-set}. 
We remark that these costs have the following properties: First, vertices in $V_1$ are only beneficial to pick, if they are picked in the first stage. 
Secondly, the only vertices, where the adversary can alter the uncertain costs are the vertices in $V_2$.
Finally, we let $\Gamma' := \Gamma$. This completes our description of the robust independent set instance.
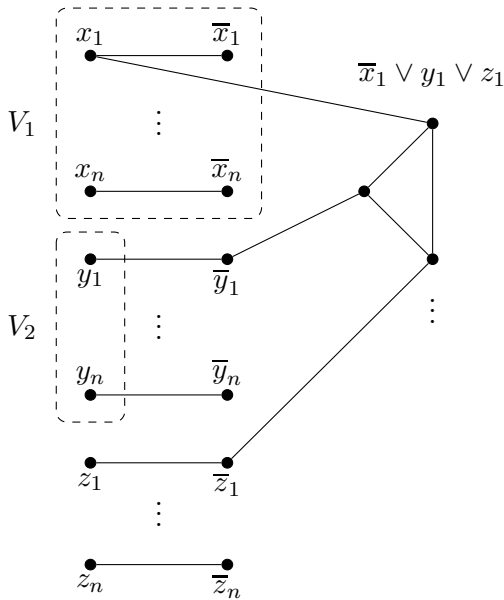
\begin{figure}[htpb]
\centering
\begin{tikzpicture}[scale=0.9, auto,swap]

\node[vertex] (zn) at (0,4.5) {};
\node[below] at (zn) {$z_n$};
\node[vertex] (notzn) at (2,4.5) {};
\node[below] at (notzn) {$\overline z_n$};
\draw[edge] (zn) to (notzn);

\node[vertex] (z1) at (0,6) {};
\node[below] at (z1) {$z_1$};
\node[vertex] (notz1) at (2,6) {};
\node[below] at (notz1) {$\overline z_1$};
\draw[edge] (z1) to (notz1);

\node[below=-14.5] at ($(z1)!0.5!(notzn)$) {$\vdots$};

\node[vertex] (yn) at (0,7) {};
\node[above] at (yn) {$y_n$};
\node[vertex] (notyn) at (2,7) {};
\node[above] at (notyn) {$\overline y_n$};
\draw[edge] (yn) to (notyn);

\node[vertex] (y1) at (0,9) {};
\node[below] at (y1) {$y_1$};
\node[vertex] (noty1) at (2,9) {};
\node[below] at (noty1) {$\overline{y}_1$};
\draw[edge] (y1) to (noty1);

\node[below=-14.5] at ($(y1)!0.5!(notyn)$) {$\vdots$};

\node[vertex] (xn) at (0,10) {};
\node[above] at (xn) {$x_n$};
\node[vertex] (notxn) at (2,10) {};
\node[above] at (notxn) {$\overline x_n$};
\draw[edge] (xn) to (notxn);

\node[vertex] (x1) at (0,12) {};
\node[above] at (x1) {$x_1$};
\node[vertex] (notx1) at (2,12) {};
\node[above] at (notx1) {$\overline x_1$};
\draw[edge] (x1) to (notx1);

\node[below=-14.5] at ($(x1)!0.5!(notxn)$) {$\vdots$};

\node[vertex] (c11) at (5,11) {};
\node[vertex] (c12) at (4,10) {};
\node[vertex] (c14) at (5,9) {};
\draw[edge] (c11) to (c12) to (c14) to (c11);
\draw[edge] (x1) to (c11);
\draw[edge] (noty1) to (c12);
\draw[edge] (notz1) to (c14);
\node[above=10] at (c11) {$\overline x_1 \lor y_1 \lor z_1$};
\node[below=5] at (c14) {$\vdots$};

\draw[dashed,rounded corners] ($(x1)+(-.5,+.7)$) rectangle ($(notxn) + (.5,-.4)$);
\node at ($(x1)+(-1,-1)$) {$V_1$};
\draw[dashed,rounded corners] ($(y1)+(-.5,+.4)$) rectangle ($(yn) + (.5,-.4)$);
\node at ($(y1)+(-1,-1)$) {$V_2$};
%\node at (1,13.5) {Variable gadgets};
%\node at (5,13.5) {Clause gadgets};
\end{tikzpicture}
\caption{Reduction from {\radj} to the robust two-stage independent set problem.}
\label{fig:max-clique-reduction}
\end{figure}

Let $\textsc{Rob}$ be the value of the robust two-stage independent set problem, that is
\begin{equation}
\textsc{Rob} = \max_{\pmb x' \in \set{0,1}^V} \min_{\pmb c \in \cU_{\Gamma'} } \max_{\pmb y' \in \X(\pmb x')} \pmb C \pmb x' + \pmb c \pmb y'. \label{eq:two-stage-ind-set}
\end{equation} 
We claim that $\textsc{Rob} \geq 3n + \ell$ if and only if the given {\radj}-instance is a Yes-instance.
\begin{table}
\centering
\begin{tabular}{l|rrr}
& $C_v$ & $\underline{c}_v$ & $\overline{c}_v$ \\
\hline
$v \in V_1$ & 1 & 0 & 0 \\
$v \in V_2$ & 0 & 1 & 0 \\
$v \not\in V_1 \cup V_2$ & 0 & 1 & 1 
\end{tabular}
\caption{Costs assigned to the vertices in the robust independent set instance}
\label{table:independent-set}
\end{table}

\textbf{Claim 1:} If $\varphi$ is a Yes-instance of {\radj}, then $\textsc{Rob} \geq 3n + \ell$. 

\emph{Proof of the claim.} Indeed, in this case there is an assignment $g_1$ of $X$-variables such that for all sets $Y' \subseteq Y$ of size at most $\Gamma$, there is an assignment $g_2$ of the variables in $Y \cup Z$ such that $Y'$ is assigned '$0$' and $\varphi$ is satisfied. 
Let $\pmb x' \in \set{0,1}^V$ be the binary vector such that for all vertices in $V_1$, $\pmb x'$ corresponds to the assignment $g_1$ and for all vertices not in $V_1$, we have $x'_v = 0$. Formally, from the two vertices $x_i$ and $\overline{x}_i$, the binary vector $\pmb x'$ includes exactly the vertex which is true under that assignment $g_1$. 
We claim that using this first-stage solution $\pmb x'$, \cref{eq:two-stage-ind-set} evaluates to at least $3n + \ell$. 
Indeed, let $V_2'$ be the set of vertices, where the adversary decreases the uncertain costs in the second stage. 
Due to the structure of the  vertex costs, the only place where costs can be decreased are vertices in $V_2$. Therefore, we can w.l.o.g.\ assume that $V_2' \subseteq V_2 = \fromto{y_1}{y_n}$. 
By the definition of $\cU_{\Gamma'}$, we have $|V'_2| \leq \Gamma' = \Gamma$. 
We now interpret $V'_2$ as a set of variables, which are forced to be '0' by the adversary.
Because $\varphi$ is a Yes-instance of {\radj}, we can find a second-stage vector $\pmb y' \in \set{0,1}^V$, such that in $\pmb x' + \pmb y'$ there is a vertex from each of the $3n$ variable gadgets and a vertex from each of the $\ell$ clause gadgets, and such that $\pmb x' + \pmb y'$ does not include any vertices where costs have been increased by the adversary (i.e\ it includes no vertices from $V'_2$). 
In total, for each $\pmb c \in \cU_\Gamma$, there exists a second-stage solution $\pmb y'$ such that $\pmb C \pmb x' + \pmb c \pmb y' \geq 3n + \ell$. 
Therefore $\textsc{Rob} \geq 3n + \ell$.

\textbf{Claim 2:} If $\textsc{Rob} \geq 3n + \ell$, then $\varphi$ is a Yes-instance of {\radj}. 

\emph{Proof of the claim.} 
In this case, there is a first-stage solution $\pmb x' \in \set{0,1}^V$, such that for all $\pmb c \in \cU_\Gamma$ there exists a second-stage solution $\pmb y' \in \set{0,1}^V$, 
such that $\pmb x'  + \pmb y'$ is indicator vector of an independent set and $\pmb C\pmb x' + \pmb c\pmb y' \geq 3n + \ell$. 
Observe that the vertex set of the graph $G(\varphi)$ can be partitioned into $3n + \ell$ cliques, where every clique corresponds to a variable gadget or a clique gadget. 
So in order to reach the total cost $\pmb C\pmb x' + \pmb c\pmb y' \geq 3n + \ell$, the solution $\pmb x' + \pmb y'$ must contain a vertex from every such clique. 
In particular, inspecting the cost structure from \cref{table:independent-set}, we see that $\pmb x'$ must contain $n$ vertices from the set $V_1$. 
As $\pmb x' \in \X$, we have that $\pmb x'$ contains either vertex $x_i$ or vertex $\overline{x}_i$ for all $i=1,\dots,n$. 
So we can define a corresponding variable assignment $g_1 : X \rightarrow \set{0,1}$. 
Using a similar reasoning as in Claim 1, we see that this variable assignment shows that $\varphi$ is a Yes-instance of {\radj}. 
This was to prove. Claim 1 and Claim 2 together complete the reduction, so we have shown that the robust two-stage independent set problem with uncertainty set $\cU_\Gamma$ is $\Sigma_3^p$-complete.
\end{proof}

\subsection{Robust Recoverable Independent Set}
\label{subsec:rec-ind-set}

The goal in this subsection is to show that the recoverable independent set problem is $\Sigma_3^p$-complete. 
The proof is very similar to the two-stage adjustable case. 
The robust recoverable independent set problem is described as follows: 
We are given a graph $G = (V,E)$ and first-stage costs $C_v \geq 0$ and second-stage cost bounds $\underline{c}_v \geq \overline{c}_v \geq 0$ for every vertex $v \in V$. 
We are also given an integer $\Gamma \geq 0$ denoting the budget and an integer $k \geq 0$ denoting the recoverability parameter. 
The task is to find the value 
\begin{equation}
 \textsc{Rob} = \max_{\pmb x  \in \X} \min_{\pmb c \in \cU_\Gamma } \max_{\pmb y \in \X(\pmb x) } \pmb C \pmb x + \pmb c \pmb y, \label{eq:recoverable-ind-set}
 \end{equation}
where $\X$ denotes the set of binary indicator vectors of independent sets, $\cU_\Gamma$ denotes the same discrete budgeted uncertainty set as in the previous subsection, and $\X(\pmb x)$ denotes the set of recovery solutions for $\pmb x$, that is $\X(\pmb x) = \set{\pmb y \in \X : \sum_i |x_i - y_i| \leq k}$.

\begin{theorem}
\label{thm:recoverable-ind-set}
Robust recoverable independent set with discrete budgeted uncertainty is $\Sigma_3^p$-complete.
\end{theorem}
\begin{proof}
The containment to the class $\Sigma_3^p$ follows from the definition, so it remains to show hardness. 
Let an instance $(\varphi, X, Y, Z)$ of {\radj} be given and let $G_0$ be the graph used in the proof of \cref{thm:discr-two-stage-ind-set}, that is, the graph from \cref{fig:max-clique-reduction}. 
Let $N_0$ be the number of vertices of $G_0$. 
We construct a new graph $G_1$ from $G_0$ by "blowing up" all the variable gadgets corresponding to $X$, in the following way: For each $i=1,\dots,n$ we delete the variable gadget for $x_i$ consisting out of vertices $x_i$ and $\overline{x_i}$ and replace it by a complete bipartite graph with $2N_0 + 2$ vertices.
 This complete bipartite graph has the $N_0+1$ vertices $x_i^{(0)},\dots,x_i^{(N_0)}$ on the left side and the $N_0+1$ vertices $\overline{x_i}^{(0)},\dots,\overline{x_i}^{(N_0)}$ on the right side of its bipartition. 
 Whenever there is an edge in the old graph $G_0$ from some vertex $v$ to $x_i$, the new graph $G_1$ contains all the edges from $v$ to $x_i^{(0)},\dots,x_i^{(N_0)}$. 
 Likewise, if there is an edge in $G_0$ from $v$ to $\overline x_i$ in the old graph, the new graph $G_1$ contains all the edges from $v$ to $\overline{x_i}^{(0)},\dots,\overline{x_i}^{(N_0)}$. 
 This completes the description of the graph $G_1$. 
 We now observe that every maximal independent set either contains all the vertices $x_i^{(0)},\dots,x_i^{(N_0)}$, or all the vertices $\overline{x_i}^{(0)},\dots,\overline{x_i}^{(N_0)}$ for each $i=1,\dots,n$. 
 We let $V_1$ be the set of all vertices of all these complete bipartite graphs, that is $V_1 = \bigcup_{i=1}^n \fromto{x_i^{(0)}}{x_i^{(N_0)}} \cup \fromto{\overline{x_i}^{(0)}}{\overline{x_i}^{(N_0)}}$. Note that $|V_1| = n(N_0+1)$. Furthermore, we let $V_2 = \fromto{y_1}{y_n}$, as in the proof of \cref{thm:discr-two-stage-ind-set}.
 We make the following claim:

\textbf{Claim:} Let $I, I'$ be two maximal independent sets in $G_1$. Then we have $|I \cap I'| \geq N_1 - N_0$ if and only if $I \cap V_1 = I' \cap V_1$.

\emph{Proof of the claim.} If the maximal independent sets $I$ and $I'$ do not agree on $V_1$, then they do not agree on at least $N_0 + 1$ vertices since they are maximal independent sets and all the "blow-up" gadgets are bipartite graphs with $N_0+1$ vertices on each side. 
On the other hand, if $I$ and $I'$ agree on $V_1$, then we have $|I \cap I'| \geq N_1 - N_0$, because the set of remaining vertices in $G_1$ without the set $V_1$ has size at most $N_0$.

Using this claim, it is straight-forward to extend the proof of \cref{thm:discr-two-stage-ind-set} to the recoverable case.
 Namely, given a {\radj} instance $(\varphi,X,Y,Z,\Gamma)$, we define an instance of the recoverable independent set problem with discrete budgeted uncertainty the following way: 
 The graph is $G_1$, the vertex costs stay the same as in \cref{table:independent-set}, the recoverability parameter is $k := 2N_0$, and the attacker's budget is $\Gamma' = \Gamma$. 
 We claim that $\varphi$ is a Yes-instance of {\radj} if and only if $\textsc{Rob} \geq 2n + n(N_0+1) + \ell$ in \cref{eq:recoverable-ind-set}.
Indeed, let $\pmb x'$ and $\pmb y'$ be two binary vectors such that $\pmb y' \in \X(\pmb x')$ is a recovery solution of $\pmb x'$ with respect to the recoverability parameter $k = 2N_0$ and such that $\pmb C \pmb x' + \pmb c \pmb y' \geq 2n + n(N_0+1) + \ell$. 
It follows from this inequality that $\pmb x' + \pmb y'$ is an indicator vector of a maximal independent set. 
Because $\pmb y'$ is a recovery solution of $\pmb x'$ we have $\sum_i |y'_i - x'_i| \leq 2N_0$. So the hamming distance of these vectors is at most $2N_0$, which means that the corresponding independent sets must have at least $N_1 - N_0$ vertices in common.
So by the claim we have that $\pmb x'$ and $\pmb y'$ agree on the vertex set $V_1$. 
The rest of the proof is analogous to the proof of \cref{thm:discr-two-stage-ind-set}.
 Informally speaking, we see that the vector $\pmb x'$ completely determines the variable assignment on the set $X$, while the vector $\pmb y'$ must agree with $\pmb x'$ on the variables $X$ but can have differing assignments of the variables $Y \cup Z$.
\end{proof}

\subsection{Robust Two-Stage Traveling Salesman Problem}
\label{subsec:two-stage-tsp}

The robust two-stage traveling salesman problem (TSP) is the robust problem of choosing a TSP tour in two stages: 
First choose a partial tour $T_1$ in the first stage, then after the reveal of the uncertain cost function $c$ choose the remainder $T_2$ of the tour. 
We wish to minimize the total cost $c(T_1 \cup T_2)$ in the worst case. 
The problem is formally defined as

\begin{equation*}
\textsc{Rob} = \min_{\pmb x \in \set{0,1}^E} \max_{\pmb c \in \cU_\Gamma } \min_{\pmb x \in \X(\pmb x)} \pmb C \pmb x + \pmb c \pmb y. 
\end{equation*} 

Here the input graph $G = (V,E)$ is a complete undirected graph, that is, $E = {V \choose 2}$.
Furthermore, $C_e \geq 0$ denotes the first-stage costs for all $e \in E$, and $\X$ denotes the set of all binary indicator vectors in $\set{0,1}^E$ such that the set of those edges with $x_e=1$ is a Hamilton cycle.
 Furthermore, $\X(\pmb x) = \set{y \in \set{0,1}^V \mid \pmb x + \pmb y \in \X}$ denotes the set of all second-stage solutions $\pmb y$ such that $\pmb y$ together with $\pmb x$ forms a Hamilton cycle. 
 To treat the case that a first-stage vector $\pmb x$ is selected which can not be completed to a Hamilton cycle, i.e. $\X(\pmb x) = \emptyset$ we define $\min \emptyset = \infty$. 
 This means the solution has the objective value $\infty$ in that case.

Finally, for given constants $\underline{c}_i, d_i \geq 0$ for all $i \in V$ and some integer $\Gamma \geq 0$, the set $\cU_\Gamma$ of uncertain cost functions is defined as
\[
\cU_\Gamma = \set{ \pmb c \in \R^V : c_i = \underline{c}_i + \delta_id_i, \delta_i \in \set{0,1}\ \forall i \in V,\ \sum_{i \in V}\delta_i \leq \Gamma }.
\]
In other words, $\cU_\Gamma$ contains those cost functions where at most $\Gamma$ entries deviate from their nominal cost $\underline{c}_i$. We define $\overline{c}_i = \underline{c}_i + d_i$. 

\begin{theorem}
\label{thm:discr-two-stage-tsp}
Robust two-stage TSP with discrete budgeted uncertainty is $\Sigma_3^p$-complete.
\end{theorem}
\begin{proof}
It follows directly form the definition that the problem is contained in the class $\Sigma_3^p$. So it remains to prove $\Sigma_3^p$-hardness. 
The starting point is the folklore reduction of 3SAT to the Hamilton cycle problem, which we only sketch here. An example is depicted in \cref{fig:hamilton-classic-reduction}. 
A \emph{variable-choice gadget} consists out of two parallel edges $x_i$ and $\overline x_i$. 
For each variable in the 3SAT instance there is a variable-choice gadget and these gadgets are all connected in a long chain. 
Furthermore, there are multiple \emph{XOR-gadgets}. An XOR-gadget between two edges $\set{a,b}$ and $\set{a',b'}$ has the effect that in every Hamilton cycle, one and only one of these two edges must be used.
For each clause $x_1 \lor x_2 \lor x_3$ in the 3SAT instance, there is a \emph{clause gadget}, which consists out of a triangle and three XOR-gadgets connecting the three edges of the triangle to the literals of that clause. 
Finally, we let the vertex set $S$ consist out of all vertices of the clause gadgets plus a vertex at the beginning and at the end of the chain of variable gadgets (marked with a square in \cref{fig:hamilton-classic-reduction}). 
All vertices of $S$ are connected in a big clique.
We claim that the described graph has a Hamilton cycle if and only if the original 3SAT formula is satisfiable. Indeed, every Hamilton cycle must choose either edge $x_i$ or $\overline x_i$ from each variable gadget. 
Suppose this choice is done in such a way that some clause is not satisfied (for example, we may choose $\overline x_1, \overline{x}_2, x_3$ in \cref{fig:hamilton-classic-reduction}). 
Then in the corresponding clause gadget, we have to choose all three edges. This is a contradiction, because a Hamilton cycle cannot contain a smaller cycle of length 3.
 
 Extending this argument, it can be seen that the graph has a Hamilton cycle if and only if the 3SAT formula is satisfiable. For every 3SAT formula $\psi$, let $G(\psi)$ denote the corresponding graph we just described. 
\begin{figure}[thpb]
\centering
\raisebox{-0.5\height}{
\begin{tikzpicture}[scale=1,auto]
\node[vertex,label=left:$a$] (x1) at (-0.5,1) {};
\node[vertex] (x2) at (0.5,1){};
\node[vertex] (x3) at (1,1){};
\node[vertex] (x4) at (1.5,1){};
\node[vertex] (x5) at (2,1){};
\node[vertex,label=right:$b$] (x6) at (3,1){};
\node[vertex] (y1) at (0.5,0.5){};
\node[vertex] (y2) at (1,0.5){};
\node[vertex] (y3) at (1.5,0.5){};
\node[vertex] (y4) at (2,0.5){};
\node[vertex,label=left:$a'$] (z1) at (-0.5,0){};
\node[vertex] (z2) at (0.5,0){};
\node[vertex] (z3) at (1,0){};
\node[vertex] (z4) at (1.5,0){};
\node[vertex] (z5) at (2,0){};
\node[vertex,label=right:$b'$] (z6) at (3,0){};
\draw[edge] (x1) -- (x2) -- (x3) -- (x4) -- (x5) -- (x6);
\draw[edge] (z1) -- (z2) -- (z3) -- (z4) -- (z5) -- (z6);
\draw[edge] (x2) -- (y1) -- (z2);
\draw[edge] (x3) -- (y2) -- (z3);
\draw[edge] (x4) -- (y3) -- (z4);
\draw[edge] (x5) -- (y4) -- (z5);
\node[vertex,label=left:$a$] (v1) at (-0.5,-2) {};
\node[vertex,label=right:$b$] (v2) at (3,-2){};
\node[vertex,label=left:$a'$] (w1) at (-0.5,-3){};
\node[vertex,label=right:$b'$] (w2) at (3,-3){};
\draw[edge] (v1) to (v2);
\draw[edge] (w1) to (w2);
\draw[XOREdge] ($(v1)!0.5!(v2)$) to coordinate (cross1) ($(w1)!0.5!(w2)$);
\node[XORGadget] at (cross1) {}; 
\end{tikzpicture}
}
\raisebox{-0.5\height}{
\begin{tikzpicture}[scale=1,auto]
\node[vertex] (s2) at (4,6.5) {};
\node[vertex] (x1) at (4,5.5) {};
\node[vertex] (x2) at (4,4) {};
\node[vertex] (x3) at (4,2.5) {};
\node[vertex] (s3) at (4,1) {};
\node[vertex] (s4) at (4,0) {};
\node[draw] at (s2) {};
\node[draw] at (s4) {};
\draw[edge, bend right] (x1) to coordinate (a1) coordinate[pos=0.25] (a1') (x2);
\node[left] at (a1) {$x_1$}; 
\draw[edge, bend left] (x1) to node[right]{$\overline x_1$} (x2);
\draw[edge, bend right] (x2) to coordinate (a2) coordinate[pos=0.25] (a2') (x3);
\node[left] at (a2) {$x_2$};
\draw[edge, bend left] (x2) to node[right]{$\overline x_2$} (x3);
\draw[edge, bend left] (x3) to coordinate (a3) coordinate[pos=0.75] (a3') (s3);
\node[right] at (a3) {$\overline x_3$};
\draw[edge, bend right] (x3) to node[left]{$x_3$} (s3);
\draw[edge] (s2) to (x1);
\draw[edge] (s3) to (s4);
\node[vertex] (c1) at (1.5,4) {};
\node[vertex] (c2) at (0.5,2.5) {};
\node[vertex] (c3) at (2.5,2.5) {};
\draw[edge] (c1) -- (c2) -- (c3) -- (c1);
\node[draw] at (c1) {};
\node[draw] at (c2) {};
\node[draw] at (c3) {};
\coordinate (cross1) at (1.5,5){}; 
\draw[XOREdge, bend right] (a1') to (cross1) to ($(c1)!0.5!(c2)$);
\node[XORGadget]  at (cross1) {}; 
\draw[XOREdge, bend right] (a2') to coordinate (cross2) ($(c1)!0.5!(c3)$);
\node[XORGadget] at (cross2) {}; 
\draw[XOREdge, bend left] (a3') to coordinate (cross3) ($(c2)!0.5!(c3)$);
\node[XORGadget] at (cross3) {}; 
\node[below=10] at (c2) {$x_1 \lor x_2 \lor \overline x_3$};
\end{tikzpicture}
}
\caption{Classical reduction of 3SAT to the Hamilton cycle problem. Arrows marked with a cross denote an XOR-gadget, as illustrated on the left. Vertices marked with a square are all connected in one clique.}
\label{fig:hamilton-classic-reduction}
\end{figure}
\begin{figure}[thpb]
\centering
\begin{tikzpicture}[scale=1,auto]
\node[vertex] (s2) at (4,6) {};
\node[vertex] (x1) at (4,5) {};
\node[vertex] (x2) at (4,4) {};
\node[vertex] (x3) at (4,3) {};
\node[vertex] (x4) at (4,2) {};
\node[vertex] (x5) at (4,1) {};
\node[vertex] (x6) at (4,0) {};
\node[vertex] (x7) at (4,-1) {};
\node[vertex] (s3) at (4,-2) {};
\node[draw] at (s2) {};
\node[draw] at (s3) {};
\draw[edge, bend right] (x1) to coordinate (a1) coordinate[pos=0.25] (a1') (x2);
\node[left] at (a1) {$x_1$}; 
\draw[edge, bend left] (x1) to node[right]{$\overline x_1$} (x2);
\draw[edge, bend right] (x2) to coordinate (a2) coordinate[pos=0.25] (a2') (x3);
\node[left] at (a2) {$x_n$};
\draw[edge, bend left] (x2) to node[right]{$\overline x_n$} (x3);
\draw[edge, bend left] (x3) to coordinate (a3) coordinate[pos=0.75] (a3') (x4);
\node[right] at (a3) {$\overline y_1$};
\draw[edge, bend right] (x3) to node[left]{$y_1$} (x4);

\draw[edge, bend right] (x4) to coordinate (a4) (x5);
\node[left] at (a4) {$y_n$}; 
\draw[edge, bend left] (x4) to node[right]{$\overline y_n$} (x5);

\draw[edge, bend right] (x5) to coordinate (a5) (x6);
\node[left] at (a5) {$z_1$}; 
\draw[edge, bend left] (x5) to node[right]{$\overline z_1$} (x6);

\draw[edge, bend right] (x6) to coordinate (a6) (x7);
\node[left] at (a6) {$z_n$}; 
\draw[edge, bend left] (x6) to node[right]{$\overline z_n$} (x7);

\node[right=12, above=-8] at (x2) {$\vdots$};
\node[left=15, above=-8] at (x2) {$\vdots$};
\node[right=12, above=-8] at (x4) {$\vdots$};
\node[left=15, above=-8] at (x4) {$\vdots$};
\node[right=12, above=-8] at (x6) {$\vdots$};
\node[left=15, above=-8] at (x6) {$\vdots$};
\draw[edge] (s2) to (x1);
\draw[edge] (s3) to (x7);

\draw[dashed,rounded corners] ($(x1)+(-1,0)$) rectangle ($(x3) + (1,0)$);
\node at ($(x2) + (-1.5,0)$) {$E_1$};
\draw[dashed,rounded corners] ($(x3)+(0,-.2)$) rectangle ($(x5) + (-1,0.2)$);
\node at ($(x4) + (-1.5,0)$) {$E_2$};

\node[vertex] (c1) at (-1,3) {};
\node[vertex] (c2) at (-2,1.5) {};
\node[vertex] (c3) at (0,1.5) {};
\draw[edge] (c1) -- (c2) -- (c3) -- (c1);
\node[draw] at (c1) {};
\node[draw] at (c2) {};
\node[draw] at (c3) {};
\node at ($(c2)!0.5!(c3) + (0,-2)$) {$\vdots$};
\coordinate (cross1) at (-1,4){}; 
\draw[edge,-{Stealth}, bend right] ($(x2) + (-3,0.5)$) to (cross1) to ($(c1)!0.5!(c2)$);
\node[XORGadget]  at (cross1) {}; 
\draw[edge,-{Stealth}, bend right] (2,3) to coordinate (cross2) ($(c1)!0.5!(c3)$);
\node[XORGadget] at (cross2) {}; 
\draw[edge,-{Stealth}, bend left] (1.5,0.5) to coordinate (cross3) ($(c2)!0.5!(c3)$);
\node[XORGadget] at (cross3) {}; 
\end{tikzpicture}
\caption{Reduction from {\radj} to the robust TSP problem.}
\label{fig:robust-tsp}
\end{figure}

In order to prove the $\Sigma_3^p$-hardness of robust two-stage TSP, we modify this reduction to work with {\radj} instead of SAT. 
Let an instance of {\radj} be given by a formula $\varphi$ in conjunctive normal form on the variable set $X \cup Y \cup Z$, and a parameter $\Gamma \geq 0$. 
Let $n = |X| = |Y| = |Z|$ be the number of variables in each set. 
By \cref{thm:gamma-SAT}, this problem is $\Sigma_3^p$-complete. 
We construct an instance of robust TSP, consisting out of a complete graph $G' = (V',E')$ with $E' = {V \choose 2}$ and edge costs $C_e, \underline{c}_e, \overline{c}_e$ for all $e \in E'$ and a parameter $\Gamma'$. 
The reduction is sketched in \cref{fig:robust-tsp}. Formally, it is described the following way:
We consider the graph $G := G(\varphi) = (V,E)$. Our TSP instance has the same vertex set $V' = V$, and all possible edges $E' = {V \choose 2}$, such that any edge $e \in E' \setminus E$ has a non-zero cost $C_e = \underline{c}_e = \overline{c}_e = 1$. 
The idea is that we will find a robust TSP tour of cost 0 if and only if $\varphi$ is a Yes-instance. A 0-cost tour will only use edges from $E$, not from $E'$.
 Furthermore, for the edges in $E$, we define the edge sets $E_1 := \fromto{x_1}{x_n} \cup \fromto{\overline{x}_1}{\overline{x}_n}$ and $E_2 = \fromto{y_1}{y_n}$.
Depending on whether an edge is in $E_1$, in $E_2$, in $E \setminus (E_1 \cup E_2)$ or none of them, we define its costs $C_e, \underline{c}_e, \overline{c}_e$ as specified in \cref{table:tsp}. 
We remark that these costs have the following properties: First, edges in $E_1$ are only beneficial to pick, if they are picked in the first stage. 
Secondly, the only edges, where the adversary can increase the uncertain costs are the edges in $E_2$.
Finally, we let $\Gamma' := \Gamma$. This completes our description of the robust TSP instance.
Let $\textsc{Rob}$ be the value of the robust two-stage TSP, that is
\begin{equation}
\textsc{Rob} = \min_{\pmb x' \in \set{0,1}^V} \max_{\pmb c \in \cU_{\Gamma'} } \min_{\pmb y' \in \X(\pmb x')} \pmb C \pmb x' + \pmb c \pmb y'. \label{eq:two-stage-tsp}
\end{equation} 
We claim that $\textsc{Rob} = 0$ if and only if the given {\radj}-instance is a Yes-instance.
\begin{table}
\centering
\begin{tabular}{l|rrr}
& $C_e$ & $\underline{c}_e$ & $\overline{c}_e$ \\
\hline
$e \in E_1$ & 0 & 1 & 1 \\
$e \in E_2$ & 1 & 0 & 1 \\
$e \in E \setminus (E_1 \cup E_2)$ & 1 & 0 & 0 \\
$e \in E' \setminus E$ & 1 & 1 & 1
\end{tabular}
\caption{Costs assigned to the edges in the robust TSP instance}
\label{table:tsp}
\end{table}

\textbf{Claim 1:} If $\varphi$ is a Yes-instance of {\radj}, then $\textsc{Rob} = 0$. 

\emph{Proof of the claim.} Indeed, in this case there is an assignment $g_1$ of $X$-variables such that for all sets $Y' \subseteq Y$ of size at most $\Gamma$, there is an assignment $g_2$ of the variables in $Y \cup Z$ such that $Y'$ is assigned '$0$' and $\varphi$ is satisfied. 
Let $\pmb x' \in \set{0,1}^E$ be the first-stage solution, which picks only edges from the set $E_1$, and in $E_1$ picks exactly those edges which correspond to the assignment $g_X$. Formally, from the two edges $x_i$ and $\overline{x}_i$, the binary vector $\pmb x'$ includes exactly the edge which is true under that assignment $g_1$. 
We claim that using this first-stage solution $\pmb x'$, \cref{eq:two-stage-tsp} evaluates to $0$. 
Indeed, let $E_2'$ be the set of vertices, where the adversary increases the uncertain costs in the second stage. 
Due to the structure of the  edge costs, the only place where costs can be increased are edges in $E_2$. Therefore, we can w.l.o.g.\ assume that $E_2' \subseteq E_2 = \fromto{y_1}{y_n}$. 
By the definition of $\cU_{\Gamma'}$, we have $|E'_2| \leq \Gamma' = \Gamma$. 
We now interpret $E'_2$ as a set of variables, which are forced to be '0' by the adversary.
Because $\varphi$ is a Yes-instance of {\radj}, we can find a second-stage vector $\pmb y' \in \set{0,1}^V$, such that in $\pmb x' + \pmb y'$ there is one edge from each of the $3n$ variable gadgets and such that $\pmb x' + \pmb y'$ describes a Hamilton cycle, and such that $\pmb x' + \pmb y'$ does not include any edge whose cost has been increased by the adversary (i.e\ it includes no edges from $V'_2$). 
In total, for each $\pmb c \in \cU_\Gamma$, there exists a second-stage solution $\pmb y'$ such that $\pmb C \pmb x' + \pmb c \pmb y' = 0$. 
Therefore $\textsc{Rob} = 0.$

\textbf{Claim 2:} If $\textsc{Rob} = 0$, then $\varphi$ is a Yes-instance of {\radj}. 

\emph{Proof of the claim.} 
In this case, there is a first-stage solution $\pmb x' \in \set{0,1}^E$, such that for all $\pmb c \in \cU_\Gamma$ there exists a second-stage solution $\pmb y' \in \set{0,1}^E$, 
such that $\pmb x'  + \pmb y'$ is indicator vector of a Hamilton cycle and $\pmb C\pmb x' + \pmb c\pmb y' = 0$. 
Then it follows from the structure of the edge costs, that $\pmb x'$ contains only edges in $E_1$, while $\pmb y'$ contains no edges from $E_1$. Because together they form a Hamilton cycle, we have that $\pmb x'$ contains exactly one edge from each $X$-variable gadget and $\pmb y'$ contains exactly one edge from each $Y$- and $Z$-variable gadget, avoiding those edges whose costs were increased by the adversary.
So we can define a variable assignment $g_1 : X \rightarrow \set{0,1}$ which corresponds exactly to $\pmb x'$. 
Using a similar reasoning as in Claim 1, we see that this variable assignment shows that $\varphi$ is a Yes-instance of {\radj}. 
This was to prove. Claim 1 and Claim 2 together complete the reduction, so we have shown that the robust two-stage independent set problem with uncertainty set $\cU_\Gamma$ is $\Sigma_3^p$-complete.
\end{proof}

\subsection{Robust Recoverable TSP}
\label{subsec:rec-tsp}

The goal in this subsection is to show that the robust recoverable TSP with discrete budgeted uncertainty is $\Sigma_3^p$-complete. 
The proof is very similar to the two-stage adjustable case. 
The robust recoverable TSP is described as follows: 
We are given a complete graph $G = (V,E)$ and first-stage costs $C_e \geq 0$ and second-stage cost bounds $0 \leq \underline{c}_e \leq \overline{c}_e$ for every edge $e \in E$. 
We are also given an integer $\Gamma \geq 0$ denoting the budget and an integer $k \geq 0$ denoting the recoverability parameter. 
The second-stage cost bounds together with $\Gamma$ define the uncertainty set $\cU_\Gamma$ as defined in the previous subsection.
The task is to find the value 
\begin{equation}
 \textsc{Rob} = \min_{\pmb x  \in \X} \max_{\pmb c \in \cU_\Gamma } \min_{\pmb y \in \X(\pmb x) } \pmb C \pmb x + \pmb c \pmb y, \label{eq:recoverable-tsp}
 \end{equation}
where $\X$ denotes the set of binary indicator vectors of Hamilton cycles, $\cU_\Gamma$ denotes the discrete budgeted uncertainty set, and $\X(\pmb x)$ denotes the set of recovery solutions for $\pmb x$, that is $\X(\pmb x) = \set{\pmb y \in \X : \sum_i |x_i - y_i| \leq k}$.

\begin{theorem}
\label{thm:recoverable-tsp}
Robust recoverable TSP with discrete budgeted uncertainty is $\Sigma_3^p$-complete.
\end{theorem}

\begin{proof}
Similar to the proof for recoverable independent set, we consider "blow ups" of the variable gadgets. 
Let $G' = G(\varphi)$ be the same graph as in the proof of \cref{thm:discr-two-stage-tsp}, i.e.\ the graph from \cref{fig:robust-tsp}. 
Let $N_0$ be the number of its vertices, and hence the length of a Hamilton cycle of this graph. We modify the graph $G'$ in the following way: 
For each of the $X$-variable gadgets belonging to $x_i$ (where $i = 1,\dots,n$), we add $N_0 + 1$ new XOR-gadgets, each of which which connects the edge $x_i$ to the edge $\overline x_i$. 
We now set the recoverability parameter $k$ to be $k = 2N_0$. Let $G''$ be the resulting graph after performing this modification for every $i=1,\dots,n$.  
It is clear that in every Hamilton cycle, for every variable $x_i$ the state of all the $N_0 + 1$ new XOR-gadgets belonging to $x_i$ is identical. 
It follows that with respect to the recoverability parameter $k$, two indicator vectors $\pmb x', \pmb y'$ of Hamilton cycles meet the condition $\pmb y' \in \X(\pmb x')$ if and only if they agree on all the $X$-variables. 
The rest of the proof is analogous to the proof of the two-stage variant,  \cref{thm:discr-two-stage-tsp}.
\end{proof}

\subsection{Two-stage and recoverable vertex cover}
\label{subsec:vertexcover}

Analogous to the previous subsections, where we considered two-stage and recoverable variants of the maximum independent set problem and the TSP, in this subsection we consider the vertex cover problem. We show that in combination with discrete budgeted uncertainty this problem is $\Sigma_3^p$-complete (both the adjustable two-stage, as well as the recoverable variant). 
Formally, if $G = (V,E)$ is a graph and $\X \subseteq \set{0,1}^V $ denotes the set of binary indicator vectors of vertex covers, then the two-stage problem is defined as

\begin{equation*}
\textsc{Rob} = \min_{\pmb x \in \set{0,1}^E} \max_{\pmb c \in \cU_\Gamma } \min_{\substack{\pmb y \in \set{0,1}^E\\ \pmb x+ \pmb y \in \X}} \pmb C \pmb x + \pmb c \pmb y 
\end{equation*} 

and the recoverable problem is defined as

\begin{equation*}
\textsc{Rob} = \min_{\pmb x \in \X} \max_{\pmb c \in \cU_\Gamma } \min_{\pmb y \in \X(\pmb x)} \pmb C \pmb x + \pmb c \pmb y. 
\end{equation*} 
Here, the cost functions and the uncertainty set are described by real numbers $C_v \geq 0$ and  $0 \leq \underline{c}_v \leq \overline{c}_v$ for every vertex $v$ and an integer $\Gamma \geq 0$. 
As the argument is very similar to the previous sections, we only provide a sketch of the proof.

\begin{theorem}
\label{thm:two-stage-and-recoverable-vertex-cover}
Robust two-stage minimum cost vertex cover with discrete budgeted uncertainty is $\Sigma_3^p$-complete. The same holds for robust recoverable vertex cover.
\end{theorem}
\begin{proof}
The proof is analogous to the proof of two-stage independent set (\cref{thm:discr-two-stage-ind-set}) and recoverable independent set (\cref{thm:recoverable-ind-set}). 
Let $(\varphi,X,Y,Z,\Gamma)$ be an instance of {\radj}, where $|X| = |Y| = |Z| = n$ and where the formula $\varphi$ consists out of $\ell$ clauses with 3 literals each.
 Let $G = G(\varphi)$ be the same graph as in the proof of \cref{thm:discr-two-stage-ind-set}. 
It is now easily seen that $G$ has a vertex cover of cardinality $2\ell + 3n$ or less if and only if $\varphi$ is satisfiable. 
Furthermore, every vertex cover requires at least that number of vertices. 
We now let $V_1, V_2$ be the same set of vertices as in the proof of \cref{thm:discr-two-stage-ind-set}, and we define vertex costs as given in \cref{table:vertex-cover}. 
We claim that an optimal solution to the two-stage vertex cover problem has robust value $\textsc{Rob} \leq 2\ell + 3n$ if and only if $\varphi$ is a Yes-Instance. 
Indeed, observe that a vertex of cost 2 can never be picked in such a solution. The rest of the argument is analogous to the proof of \cref{thm:discr-two-stage-ind-set}. 
This shows that two-stage vertex cover is $\Sigma_3^p$-complete. 
Finally, the modification described in the proof of \cref{thm:recoverable-ind-set} (blowing up every $X$-variable gadget by a factor $N_0+1$, where $N_0$ is the original size of $G$) works in the same manner. 
This shows that recoverable vertex cover is $\Sigma_3^p$-complete.
\begin{table}
\centering
\begin{tabular}{l|rrr}
& $C_e$ & $\underline{c}_e$ & $\overline{c}_e$ \\
\hline
$e \in V_1$ & 1 & 2 & 2 \\
$e \in V_2$ & 2 & 1 & 2 \\
$e \in V \setminus (V_1 \cup V_2)$ & 2 & 1 & 1 
\end{tabular}
\caption{Costs assigned to the vertices in the robust vertex cover instance}
\label{table:vertex-cover}
\end{table}
\end{proof}

\subsection{Multi-stage versions}
\label{subsec:multistage}

In the last part of this section, we show that all our proofs for hardness of robust two-stage adjustable optimization can also be extended to the case where instead of two stages, we have an arbitrary number $K$ of stages. For $K \geq 1$, we define the \emph{robust $K$-stage adjustable optimization problem} with discrete budgeted uncertainty sets the following way: The input consists out of first-stage costs $\pmb c^{(0)}$ and $K-1$ independent discrete budgeted uncertainty sets $\cU^{(1)}_{\Gamma_1},\dots,\cU^{(K-1)}_{\Gamma_{K-1}}$. Let furthermore $\cB := \set{0,1}^n$ denote the set of all possible binary vectors and $\X \subseteq \cB$ denote the set of all feasible solutions (for example all vertex covers).
The task is to find the value

\begin{equation*}
\textsc{Rob} = \min_{\pmb x^{(0)} \in \cB} \max_{\pmb c^{(1)} \in \cU^{(1)}_{\Gamma_1} } \min_{\pmb x^{(1)} \in \cB} \dots \min_{\pmb x^{(K-2)} \in \cB} \max_{\pmb c^{(K-1)} \in \cU^{(K-1)}_{\Gamma_{K-1}}} \min_{\substack{\pmb x^{(K-1)} \in \cB\\ \sum_i \pmb x^{(i)} \in \X}} \sum_{i=0}^{K-1} \pmb c^{(i)}\pmb x^{(i)}.
\end{equation*} 

Note that for $K=2$, we have exactly the two-stage adjustable problem. 
\begin{theorem}
If $K \geq 2$ is a constant, then the $K$-stage versions of TSP, independent set, and vertex cover in combination with discrete budgeted uncertainty are $\Sigma_{2K-1}^p$-complete. If $K$ is part of the input, these problems are PSPACE-complete.
\end{theorem}
\begin{proof}
The proof can easily be adapted from the previous proofs of \cref{thm:discr-two-stage-ind-set,thm:recoverable-ind-set,thm:discr-two-stage-tsp,thm:recoverable-tsp,thm:two-stage-and-recoverable-vertex-cover}.
We make use of the fact that $K$-stage {\radj} is $\Sigma_{2K-1}^p$-complete for constant $K$ and PSPACE-complete for $K$ part of the input (\cref{thm:multi-stage-gamma-sat}). 
\end{proof}

\section{Conclusions}
\label{sec:conclusions}

Multi-stage (adjustable) robust optimization is a natural extension of static, one-stage approaches to model a more dynamic decision making environment. By giving an opportunity to react to adversarial choices, it is possible to reach better objective values and thus to reduce the conservatism of robust optimization. The benefits, however, come with increased computational difficulties.

While several heuristic and exact solution methods have been developed, the complexity of many such problems remained open. Of particular importance is whether a problem still remains in NP, and thus allows for a compact mixed-integer programming formulation, or if it has a higher complexity in the polynomial hierarchy. 

In this paper we first introduced a variant of a multi-stage satisfiability problem, where the adversary has a budget on the number of ''attacks'' (forcing variables to zero). This problem is designed to capture the key difficulties of protecting against budgeted uncertainty sets, where a bound on the deviation from the nominal scenario is used. With the help of this SAT problem, we are able to show that adjustable problems with continuous budgeted uncertainty in the right-hand side are $\Sigma^p_3$-complete, while the problem remains in NP if the uncertainty is in the objective.

We then considered a range of classic combinatorial optimization problems (independent set, traveling salesman, vertex cover) under discrete uncertainty sets and showed that these problems become $\Sigma^p_3$-complete as well. By natural extension, $K$-stage problem variants become $\Sigma^p_{2K-1}$-complete.

For future research, a lot of open questions remain. In this work, we showed that it is often times $\Sigma^p_3$-hard to compute an exact solution to a robust multi-stage problem. More generally, one could also examine when it is $\Sigma^p_3$-hard to even give a constant-factor approximation.
Secondly, one could strengthen \cref{thm:mixed-binary-recourse} (for continuous budgeted uncertainty) by showing hardness for more restricted optimization problems which are special cases of our model.
Finally, it would be very insightful to find some sort of ''meta-theorem'', which generalizes the results from \cref{sec:discbudgeted} (for discrete budgeted uncertainty). 
Such a meta-theorem may introduce some easy-to-fulfill property such that for all nominal optimization problems with this property, the corresponding robust two-stage adjustable (or robust recoverable, respectively)  counterpart is $\Sigma^p_3$-hard. 
Such a meta-theorem would show that in fact a lot more problems than just TSP, independent set and vertex cover possess the properties showcased in \cref{sec:discbudgeted}. As our methods are quite general, we deem it likely for such a meta-theorem to exist.


\begin{thebibliography}{BTGGN04}

\bibitem[ABV09]{aissi2009min}
Hassene Aissi, Cristina Bazgan, and Daniel Vanderpooten.
\newblock Min--max and min--max regret versions of combinatorial optimization
  problems: A survey.
\newblock {\em European journal of operational research}, 197(2):427--438,
  2009.

\bibitem[BB09]{bertsimas2009constructing}
Dimitris Bertsimas and David~B Brown.
\newblock Constructing uncertainty sets for robust linear optimization.
\newblock {\em Operations research}, 57(6):1483--1495, 2009.

\bibitem[BG20]{bold2020}
Matthew Bold and Marc Goerigk.
\newblock Recoverable robust single machine scheduling with polyhedral
  uncertainty, 2020.

\bibitem[BK17]{buchheim2017min}
Christoph Buchheim and Jannis Kurtz.
\newblock Min--max--min robust combinatorial optimization.
\newblock {\em Mathematical Programming}, 163(1):1--23, 2017.

\bibitem[BS03]{bertsimas2003robust}
Dimitris Bertsimas and Melvyn Sim.
\newblock Robust discrete optimization and network flows.
\newblock {\em Mathematical programming}, 98(1):49--71, 2003.

\bibitem[BS04]{bertsimas2004price}
Dimitris Bertsimas and Melvyn Sim.
\newblock The price of robustness.
\newblock {\em Operations research}, 52(1):35--53, 2004.

\bibitem[BTEGN09]{ben2009robust}
Aharon Ben-Tal, Laurent El~Ghaoui, and Arkadi Nemirovski.
\newblock {\em Robust optimization}, volume~28.
\newblock Princeton university press, 2009.

\bibitem[BTGGN04]{ben2004adjustable}
Aharon Ben-Tal, Alexander Goryashko, Elana Guslitzer, and Arkadi Nemirovski.
\newblock Adjustable robust solutions of uncertain linear programs.
\newblock {\em Mathematical programming}, 99(2):351--376, 2004.

\bibitem[BTN02]{ben2002robust}
Aharon Ben-Tal and Arkadi Nemirovski.
\newblock Robust optimization--methodology and applications.
\newblock {\em Mathematical programming}, 92(3):453--480, 2002.

\bibitem[B{\"u}s12]{busing2012recoverable}
Christina B{\"u}sing.
\newblock Recoverable robust shortest path problems.
\newblock {\em Networks}, 59(1):181--189, 2012.

\bibitem[CCLW16]{caprara2016bilevel}
Alberto Caprara, Margarida Carvalho, Andrea Lodi, and Gerhard~J Woeginger.
\newblock Bilevel knapsack with interdiction constraints.
\newblock {\em INFORMS Journal on Computing}, 28(2):319--333, 2016.

\bibitem[CGKZ18]{chassein2018recoverable}
Andr{\'e} Chassein, Marc Goerigk, Adam Kasperski, and Pawe{\l} Zieli{\'n}ski.
\newblock On recoverable and two-stage robust selection problems with budgeted
  uncertainty.
\newblock {\em European Journal of Operational Research}, 265(2):423--436,
  2018.

\bibitem[CH17]{chistikov2017complexity}
Dmitry Chistikov and Christoph Haase.
\newblock On the complexity of quantified integer programming.
\newblock In {\em 44th International Colloquium on Automata, Languages, and
  Programming (ICALP 2017)}. Schloss Dagstuhl-Leibniz-Zentrum fuer Informatik,
  2017.

\bibitem[CS20]{claus2020note}
Matthias Claus and Maximilian Simmoteit.
\newblock A note on {$\Sigma^2_p$}-completeness of a robust binary linear
  program with binary uncertainty set.
\newblock {\em Operations Research Letters}, 48(5):594--598, 2020.

\bibitem[CSN22]{coco2022robust}
Amadeu~A Coco, Andr{\'e}a~Cynthia Santos, and Thiago~F Noronha.
\newblock Robust min-max regret covering problems.
\newblock {\em Computational Optimization and Applications}, 83(1):111--141,
  2022.

\bibitem[DW10]{deineko2010pinpointing}
Vladimir~G Deineko and Gerhard~J Woeginger.
\newblock Pinpointing the complexity of the interval min--max regret knapsack
  problem.
\newblock {\em Discrete Optimization}, 7(4):191--196, 2010.

\bibitem[GH21]{goerigk2021multistage}
Marc Goerigk and Michael Hartisch.
\newblock Multistage robust discrete optimization via quantified integer
  programming.
\newblock {\em Computers \& Operations Research}, 135:105434, 2021.

\bibitem[GJ79]{garey1979computers}
Michael~R Garey and David~S Johnson.
\newblock {\em Computers and intractability}.
\newblock W. H. Freeman, 1979.

\bibitem[GLS81]{grotschel1981ellipsoid}
Martin Gr{\"o}tschel, L{\'a}szl{\'o} Lov{\'a}sz, and Alexander Schrijver.
\newblock The ellipsoid method and its consequences in combinatorial
  optimization.
\newblock {\em Combinatorica}, 1(2):169--197, 1981.

\bibitem[GLW22]{goerigk2022recoverable}
Marc Goerigk, Stefan Lendl, and Lasse Wulf.
\newblock Recoverable robust representatives selection problems with discrete
  budgeted uncertainty.
\newblock {\em European Journal of Operational Research}, 2022.

\bibitem[Gr{\"u}22]{grune2022complexity}
Christoph Gr{\"u}ne.
\newblock The complexity classes of hamming distance recoverable robust
  problems.
\newblock {\em arXiv preprint arXiv:2209.06939}, 2022.

\bibitem[Har20]{phdhartisch}
Michael Hartisch.
\newblock {\em Quantified integer programming with polyhedral and
  decision-dependent uncertainty}.
\newblock PhD thesis, University of Siegen, 2020.

\bibitem[HKW15]{hanasusanto2015k}
Grani~A Hanasusanto, Daniel Kuhn, and Wolfram Wiesemann.
\newblock K-adaptability in two-stage robust binary programming.
\newblock {\em Operations Research}, 63(4):877--891, 2015.

\bibitem[KZ11]{kasperski2011approximability}
Adam Kasperski and Pawe{\l} Zieli{\'n}ski.
\newblock On the approximability of robust spanning tree problems.
\newblock {\em Theoretical Computer Science}, 412(4-5):365--374, 2011.

\bibitem[KZ16]{kasperski2016robust}
Adam Kasperski and Pawe{\l} Zieli{\'n}ski.
\newblock Robust discrete optimization under discrete and interval uncertainty:
  A survey.
\newblock In {\em Robustness analysis in decision aiding, optimization, and
  analytics}, pages 113--143. Springer, 2016.

\bibitem[KZ17]{kasperski2017robust}
Adam Kasperski and Pawe{\l} Zieli{\'n}ski.
\newblock Robust recoverable and two-stage selection problems.
\newblock {\em Discrete Applied Mathematics}, 233:52--64, 2017.

\bibitem[LK10]{lodwick2010fuzzy}
Weldon~A Lodwick and Janusz Kacprzyk.
\newblock {\em Fuzzy optimization: Recent advances and applications}, volume
  254 of {\em Studies in Fuzziness and Soft Computing}.
\newblock Springer, 2010.

\bibitem[LLMS09]{liebchen2009concept}
Christian Liebchen, Marco L{\"u}bbecke, Rolf M{\"o}hring, and Sebastian
  Stiller.
\newblock The concept of recoverable robustness, linear programming recovery,
  and railway applications.
\newblock In {\em Robust and online large-scale optimization}, pages 1--27.
  Springer, 2009.

\bibitem[NCH22]{nabli2022complexity}
Adel Nabli, Margarida Carvalho, and Pierre Hosteins.
\newblock Complexity of the multilevel critical node problem.
\newblock {\em Journal of Computer and System Sciences}, 127:122--145, 2022.

\bibitem[NP20]{nguyen2020computational}
Danny Nguyen and Igor Pak.
\newblock The computational complexity of integer programming with
  alternations.
\newblock {\em Mathematics of Operations Research}, 45(1):191--204, 2020.

\bibitem[Pow19]{powell2019unified}
Warren~B Powell.
\newblock A unified framework for stochastic optimization.
\newblock {\em European Journal of Operational Research}, 275(3):795--821,
  2019.

\bibitem[PS21]{pfetsch2021generic}
Marc~E Pfetsch and Andreas Schmitt.
\newblock A generic optimization framework for resilient systems.
\newblock {\em Optimization Online}, 2021.

\bibitem[Sto76]{stockmeyer1976polynomial}
Larry~J Stockmeyer.
\newblock The polynomial-time hierarchy.
\newblock {\em Theoretical Computer Science}, 3(1):1--22, 1976.

\bibitem[SU02]{schaefer2002completeness}
Marcus Schaefer and Christopher Umans.
\newblock Completeness in the polynomial-time hierarchy: A compendium.
\newblock {\em SIGACT news}, 33(3):32--49, 2002.

\bibitem[Woe21]{woeginger2021trouble}
Gerhard~J Woeginger.
\newblock The trouble with the second quantifier.
\newblock {\em 4OR}, 19(2):157--181, 2021.

\bibitem[YGdH19]{yanikouglu2019survey}
{\.I}hsan Yan{\i}ko{\u{g}}lu, Bram~L Gorissen, and Dick den Hertog.
\newblock A survey of adjustable robust optimization.
\newblock {\em European Journal of Operational Research}, 277(3):799--813,
  2019.

\bibitem[ZZ13]{zeng2013solving}
Bo~Zeng and Long Zhao.
\newblock Solving two-stage robust optimization problems using a
  column-and-constraint generation method.
\newblock {\em Operations Research Letters}, 41(5):457--461, 2013.

\end{thebibliography}
\end{document}